\documentclass[11pt]{amsart}
\usepackage[margin=1in]{geometry}
\usepackage[utf8]{inputenc}
\usepackage{amsthm, amssymb}
\usepackage[colorlinks]{hyperref}
\usepackage[colorinlistoftodos]{todonotes}
\usepackage{algorithm}
\usepackage{algpseudocode}
\usepackage{pgf}
\usepackage{tikz}
\usepackage{tikz-cd}
\usetikzlibrary{positioning,shapes,shadows,arrows}
\usepackage{bbm,bm}
\usepackage{enumitem}

\usepackage{ytableau}

\newtheorem{prop}{Proposition}[section]
\newtheorem{thm}[prop]{Theorem}

\newtheorem{lemma}[prop]{Lemma}
\newtheorem{cor}[prop]{Corollary}
\newtheorem{conj}[prop]{Conjecture}
\newtheorem{exa}[prop]{Example}
\newtheorem{rem}[prop]{Remark}
\newtheorem{defn}[prop]{Definition}

\newcommand{\NN}{\mathbb{N}}

\newcommand{\df}{:=}

\newcommand{\fL}{\mathfrak{L}}

\newcommand{\KD}{\mathsf{KD}}
\newcommand{\KKD}{\mathsf{KKD}}
\newcommand{\KT}{\mathsf{KT}}

\newcommand{\key}{\mathsf{key}}
\newcommand{\X}{\mathsf{X}}
\newcommand{\la}{\lambda}
\newcommand{\Label}{\mathsf{Label}}
\newcommand{\LL}{\mathsf{L}}
\newcommand{\KL}{\mathsf{K_-}}

\newcommand{\RSVT}{\mathsf{RSVT}}
\newcommand{\RSSYT}{\mathsf{RSSYT}}

\newcommand{\supp}{\mathsf{supp}}

\newcommand{\wt}{\mathsf{wt}}

\newcommand{\ex}{\mathsf{ex}}

\newcommand\setItemnumber[1]{\setcounter{enumi}{\numexpr#1-1\relax}}

\definecolor{darkred}{rgb}{0.7,0,0} % darkred color

\newcommand{\definition}[1]{{\color{darkred}\emph{#1}}} % emphasis of a definition

\title{A bijection between $K$-Kohnert diagrams and reverse set-valued tableaux}

\author[J.~Pan]{Jianping Pan}
\address[J. Pan]{Department of Mathematics, NC State University, Raleigh, NC 95616-8633, U.S.A.}
\email{jpan9@ncsu.edu}

\author[T.~Yu]{Tianyi Yu}
\address[T. Yu]{Department of Mathematics, UC San Diego, La Jolla, CA 92093, U.S.A.}
\email{tiy059@ucsd.edu}

\date{\today}

\begin{document}
\maketitle
\begin{abstract}
Lascoux polynomials are $K$-theoretic analogues of the key polynomials. 
They both have combinatorial formulas involving tableaux:
reverse set-valued tableaux ($\RSVT$) rule for Lascoux polynomials
and reverse semistandard Young tableaux ($\RSSYT$) rule for key polynomials.
Furthermore, key polynomials have a simple algorithmic model in terms of Kohnert diagrams,
which are in bijection with $\RSSYT$.
Ross and Yong introduced $K$-Kohnert diagrams, 
which are analogues of Kohnert diagrams.
They conjectured a $K$-Kohnert diagram rule for Lascoux polynomials.
We establish this conjecture by constructing a weight-preserving bijection 
between $\RSVT$ and $K$-Kohnert diagrams.
\end{abstract}

\section{Introduction}
\label{section.intro}
Fix a positive integer $n$ throughout this paper. 
A \definition{weak composition of length $n$} 
is a sequence of $n$ non-negative integers.
If $\alpha$ is a weak composition,
we use $\alpha_i$ to denote its $i^{th}$ entry. 

Key polynomials $\kappa_\alpha$ are homogeneous polynomials labeled by weak compositions. They were first introduced by Demazure~\cite{De} as the characters of the Demazure modules. Further studies~\cite{Ko,LS1,LS2,RS,Kashiwara,L2,Littelmann,A,A2} provided several combinatorial formulas.

Lascoux polynomials $\fL^{(\beta)}_\alpha$ are $K$-theoretic generalizations of key polynomials~\cite{L}. They are inhomegeneous polynomials with an extra variable $\beta$. Setting $\beta = 0$
in $\fL^{(\beta)}_\alpha$ yields $\kappa_\alpha$. There are several existing combinatorial formulas for $\fL^{(\beta)}_\alpha$ involving set-valued skyline fillings and set-valued tableaux~\cite{BSW, Y}.
In this paper, we will define Lascoux polynomials by a combinatorial formula 
involving reverse set-valued tableaux ($\RSVT$).
It first appeared implicitly in~\cite{BSW} 
and was rediscovered by Shimozono and the second author~\cite{SY}.
Specifically, for each weak composition $\alpha$,
there is a set $\RSVT(\alpha)$, 
which consists of certain $\RSVT$ satisfying a left key condition 
(see subsection~\ref{background.tableaux}). 
Then $\fL^{(\beta)}_\alpha$ can be written as a sum over $\RSVT(\alpha)$:
\[
    \fL^{(\beta)}_\alpha \df \sum_{T\in \RSVT(\alpha)} \beta^{\ex(T)}x^{\wt(T)}\,.
\]

Ross and Yong~\cite{RossY} defined a generalization of 
Kohnert's move on diagrams~\cite{Ko}. 
We call them $K$-Kohnert moves. 
Repeatedly applying K-Kohnert moves on the key diagram of~$\alpha$ 
yields a set of diagrams, which is denoted as $\KKD(\alpha)$
(see subsection~\ref{background.diagrams}).

\begin{conj}~\cite{RossY} \label{Main conjecture} 
The Lascoux polynomials indexed by $\alpha$, is given by
\[
\fL^{(\beta)}_\alpha = \sum_{D \in \KKD(\alpha)}\beta^{\ex(D)}\bm{x}^{\wt(D)}\,.
\]
\end{conj}

Pechenik and Scrimshaw~\cite{PS} proved a special case
of this conjecture 
where all positive numbers in $\alpha$ 
are the same.
This paper will prove the conjecture for all $\alpha$.
\begin{thm}
Conjecture~\ref{Main conjecture} is true.
\end{thm}

To prove this theorem, 
we define two maps: $\Psi_\alpha$ on $\KKD(\alpha)$
(see subsection~\ref{prelim.kkd.to.rsvt})
and $\Phi_\alpha$ on $\RSVT(\alpha)$ 
(see subsection~\ref{prelim.rsvt.to.kkd}).
We will show $\Psi_\alpha$ (resp. $\Phi_\alpha$)
is a well-defined map to $\RSVT(\alpha)$ 
(resp. $\KKD(\alpha)$).
Finally, we establish the following.

\begin{thm}
\label{T: Main}
The maps $\Psi_\alpha: \KKD(\alpha) \to \RSVT(\alpha)$ 
and $\Phi_\alpha: \RSVT(\alpha) \to \KKD(\alpha)$
are mutually inverses of each other.
Moreover, they preserve $\wt(\cdot)$ and $\ex(\cdot)$.
\end{thm}

The paper is organized as follows. 
In section~\ref{section.new.background}, 
we review related combinatorial rules 
for $\kappa_\alpha$ and $\fL^{(\beta)}_\alpha$. 
In section~\ref{section.maps}, 
we define two maps $\Psi_\alpha$ and $\Phi_\alpha$ 
on $\KKD(\alpha)$ and $\RSVT(\alpha)$ respectively. 
The following sections will prove Theorem~\ref{T: Main}.
In section~\ref{section.Bruhat}, 
we introduce a partial order on all weak compositions. 
We call it the Bruhat order and show it is equivalent to the left swap order in~\cite{A}. 
In section~\ref{section.describe KKD RSVT}, 
we describe the sets $\KKD(\alpha)$ and $\RSVT(\alpha)$
recursively using the Bruhat order. 
In section~\ref{section.flat.sharp}, we introduce two auxiliary operators 
$\sharp_g$ and $\flat_e$ on $\KD(\alpha)$ and discuss their properties. 
In section~\ref{section.recursive}, 
we give recursive descriptions of maps $\Psi_\alpha$ and $\Phi_\alpha$ 
in terms of $\sharp_g$ and $\flat_e$.
Then we show $\Psi_\alpha$ (resp. $\Phi_\alpha$) is 
a well-defined map to $\RSVT(\alpha)$ (resp. $\KKD(\alpha)$)
using the recursive definitions developed in section~\ref{section.describe KKD RSVT}.
Finally we prove Theorem~\ref{T: Main}.

\section{Background}\label{section.new.background}
\subsection{$\RSSYT(\alpha)$ and $\RSVT(\alpha)$}\label{background.tableaux}
Given a partition $\la = (\la_1\geqslant \la_2\geqslant \dots \geqslant \la_\ell \geqslant 0)$, a \definition{Young diagram} of shape $\la$ is a finite collection of boxes, aligned at the left, in which the $i^{th}$ row has $\la_i$ boxes. 
We use English convention for our Young diagrams and tableaux,
so the first row is the highest row. 

A \definition{reverse semistandard Young tableau} of shape $\la$ is a filling of the Young diagram $\la$ with positive number such that
\begin{enumerate}[label=(\roman*)]
    \item each box contains exactly one number,
    \item the entries in each row weakly decrease from left to right, and
    \item the entries of each column strictly decrease from top to bottom.
\end{enumerate}
Let $\RSSYT_\lambda$ be the set of all the reverse semistandard Young tableaux of shape $\la$.

Following~\cite{B:Gr},
we introduce another set of tableaux
where a box might have more than one number.
A \definition{reverse set-valued tableau} of shape $\la$ is a filling of the Young diagram $\la$ with positive numbers such that
\begin{enumerate}[label=(\roman*)]
    \item each box contains a finite and non-empty set of positive integers,
    \item if a set $A$ is to the left of a set $B$ in the same row, then $\min(A) \geqslant \max(B)$, and
    \item if a set $C$ is below a set $A$ in the same column, then $\min(A) > \max(C)$.
\end{enumerate}
Let $\RSVT_\lambda$ be all the reverse set-valued tableaux of shape $\la$.

Let the \definition{weight vector} for $T$ be the weak composition whose $i^{th}$ component is the the total number of appearance of $i$ in $T$, denoted by $\wt(T)$. Given any weak composition $\alpha$, let $|\alpha| = \sum_{i\geqslant 1}\alpha_i$. Given $T\in \RSVT_\lambda$, define $\LL(T)$ to be an element in $\RSSYT_\la$ constructed by only keeping the largest number in each box of $T$. We call these numbers the \definition{leading numbers} of $T$. Any number in $T$ that is not a leading number is called a \definition{extra number}. Let the \definition{excess} of $T$ be the number of extra numbers in $T$, so we can denote it by $\ex(T) = |\wt(T)|-|\la|$.

Next we give the definition of \definition{left key} of $T$, denoted by $\KL(T)$, where $T$ is a $\RSSYT$. It was first given in ~\cite[Section~5]{W}. We give the description as in~\cite[Definition~3.11]{SY}.

\begin{defn} \label{def:left_key} 
Let $C_1,C_2$ be two adjacent columns from a $\RSSYT$ with $C_1$ on the left. 
We may view $C_1$ and $C_2$ as sets. 
We define $C_1 \lhd C_2$ as follows. Assume $C_2 = \{a_1<a_2<\dots <a_m\}$. Start by finding the smallest $b_1\in C_1$ such that $b_1 \geqslant a_1$. Then find the smallest $b_2\in C_2$ such that $b_2 \geqslant a_2$ and $b_2>b_1$. Continue until we find all $b_1,b_2,\dots, b_m$. Then $C_1\lhd C_2 \df b_1<b_2<\dots b_m$. Let $C_1,\dots C_k$ be $k$ columns in a $\RSSYT$, then we can define recursively,
\[
C_1 \lhd C_2 \dots \lhd C_k \df C_1 \lhd (C_2\lhd \dots \lhd C_k).
\]

Given a reverse semistandard Young tableau $T$ with columns $C_1,C_2,\dots,C_n$. Then its left key $\KL(T)$ is a $\RSSYT$ constructed by taking $C_1 \lhd C_2 \dots \lhd C_k$ as its $k^{th}$ column.

Given a reverse set-valued tableau $T$, 
its left key $K_-(T)$ is defined as $K_-(\LL(T))$.
\end{defn}

\begin{exa}\label{eg_left_key}Consider the following $T\in \RSVT_{(3,2)}$. 
We have $\wt(T) = (2,2,2,1,0,1)$ and \linebreak $\ex(T) = 3$. We can also compute $\LL(T)$ and its left key.
\[
T = \begin{ytableau}
64 & 32 & 2\cr
31 & 1 \cr
\end{ytableau}\,, \quad
\LL(T) = \begin{ytableau}
6 & 3 & 2\cr
3 & 1 \cr
\end{ytableau}\,, \quad
\KL(T) = 
\KL(\LL(T)) = \begin{ytableau}
6 & 6 & 6\cr
3 & 3 \cr
\end{ytableau}\,.
\]
\end{exa} 

Given a weak composition $\alpha = (\alpha_1,\alpha_2,\dots,\alpha_n)$, let $\alpha^+$ be the partition obtained from $\alpha$ by sorting the numbers in decreasing order and ignoring the trailing $0$'s. Define the \definition{key tableau} for $\alpha$ to be the unique element in $\RSSYT_{\alpha^+}$ whose $j^{th}$ column consists of the numbers $\{i |\, \alpha_i \geqslant j\}$. Denote this tableau by $\key(\alpha)$.

\begin{rem}
For any reverse set-valued tableau $T$, 
$\KL(T)$ is a key tableau.
\end{rem}

With above concepts, we now define the subsets of $\RSSYT_{\alpha^+}$ and $\RSVT_{\alpha^+}$ that will be used to compute $\kappa_\alpha$ and $\fL^{(\beta)}_\alpha$.

\begin{align*}
\RSSYT(\alpha) &\df \{T\in \RSSYT_{\alpha^+}: \KL(T) \leqslant \key(\alpha)\}\\
\RSVT(\alpha) &\df \{T\in \RSVT_{\alpha^+}: \KL(T) \leqslant \key(\alpha)\}
\end{align*}
Here the $\leqslant$ relation means entry-by-entry comparison. For example, $T$ in Example~\ref{eg_left_key} is in $\RSVT((0,0,2,0,0,3))$ but not in $\RSVT((0,2,0,0,3))$.

We now list the combinatorial formulas in~\cite{LS1,LS2,RS} for \definition{key polynomials}, and in~\cite{BSW,SY} for \definition{Lascoux polynomials} labeled by a weak composition $\alpha$:
\[
    \kappa_\alpha \df \sum_{T\in \RSSYT(\alpha)} x^{\wt(T)}\,, \quad
    \fL^{(\beta)}_\alpha \df \sum_{T\in \RSVT(\alpha)} \beta^{\ex(T)}x^{\wt(T)}\,.
\]

\subsection{Viewing $\RSVT(\alpha)$ 
as a pair of diagrams}
In this subsection, we introduce another perspective
on $\RSVT(\alpha)$.
A \definition{diagram} is a finite subset 
of $\NN \times \NN$.
We may represent a diagram by putting a box at row $r$
and column $c$ for each $(c,r)$ in the diagram.
We adopt the convention where columns begin at $1$
from the left and rows begin at $1$ from the bottom.
The \definition{weight} of a diagram $D$, 
denoted as $\wt(D)$, 
is a weak composition 
whose $i^{th}$ entry is the number of boxes 
in its $i^{th}$ row.

A \definition{diagram pair} is 
an ordered pair $D = (D_1, D_2)$
such that $D_1$ and $D_2$ are disjoint diagrams.
We may represent $D$ by putting a box 
at $(c,r)$ for each $(c,r) \in D_1$ and putting a box 
with label $\X$ at $(c,r)$ for each $(c,r) \in D_2$. 
Cells in $D_1$ are called \definition{Kohnert cells}.
Cells in $D_2$ are called \definition{ghost cells}.
The \definition{weight} of $D$, 
denoted as $\wt(D)$, 
is a weak composition 
whose $i^{th}$ entry is the number of Kohnert cells and ghost cells 
in its $i^{th}$ row.
Let the \definition{excess} of $D$, 
denoted by $\ex(D)$, be $|D_2|$.

Now we embed the set of $\RSVT$ 
into the set of diagram pairs.
Given an $\RSVT$ $T$, 
we send it to $(L, E)$.
The set $L$ (resp. $E$) consists of all $(r, c)$
such that $r$ is a leading (resp. extra) number 
in column $c$ of $T$.
This map is injective.
If we know $(L, E)$ is the image of some $\RSVT$ $T$,
we can uniquely recover $T$:
First, for each $c$, 
build a column that consists of $r$
such that $(c, r) \in L$.
The column should be decreasing from top to bottom.
Then for each $(c,r) \in E$,
put $r$ in the lowest cell whose largest number
is larger than $r$.
This will be column $c$ of $T$.
Now we may view each $RSVT$ as a diagram pair.
We write $T = (L, E)$
to denote this correspondence.
It is clear that this correspondence preserves
$\wt(\cdot)$ and $\ex(\cdot)$.

\begin{exa} \label{RSVT as diagram pair}
Let $T$ be the RSVT in the previous example. 
It corresponds to the diagram pair
$(\{(1,3), (1,6), (2,1), (2,3), (3,2) \}, \{ (1,1), (1,4), (2,2)\})$, which can be presented as
$$
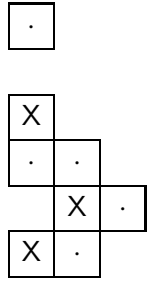
\begin{ytableau}
\cdot \\
\none \\
\X \\
\cdot & \cdot \\
\none & \X & \cdot \\
\X & \cdot
\end{ytableau}
$$
Viewing $T$ as a diagram pair, 
we have $\wt(T) = (2,2,2,1,0,1)$ and $\ex(T) = 3$,
which agrees with the previous example. 
\end{exa}

We may also view $\RSSYT(\alpha)$ as a subset of $\RSVT(\alpha)$.
Thus, $\RSSYT(\alpha)$ is the set of diagram pairs $(L, \emptyset) \in \RSVT(\alpha)$.
With this convention, we have the following observation.
\begin{rem}
If the diagram pair $(L, E)$ is $\RSVT(\alpha)$,
then $(L, \emptyset) \in \RSSYT(\alpha)$.
\end{rem}

\subsection{$\KD(\alpha)$ and $\KKD(\alpha)$}\label{background.diagrams}
We give another combinatorial definition of key polynomials due to Kohnert~\cite{Ko}. 
A diagram pair is called a \definition{key diagram pair}
if its Kohnert cells are left-justified and has no ghost cells.
Given a weak composition $\alpha$, 
we let $D_\alpha$ the key diagram pair associated to $\alpha$:
On its row $i$, there are $\alpha_i$ left-justified Kohnert cells 
and no ghost cells. 

Next, we define a \definition{Kohnert move} on a diagram pair
with no ghost cells:
Select the rightmost box in any row and move it downward to the first position available, possibly jumping over other cells as needed. Let $\KD(\alpha)$ be the closure of $\{D_\alpha\}$ under all possible Kohnert moves.

\begin{thm}~\cite{Ko}
The key polynomials indexed by $\alpha$, is given by 
\[
\kappa_\alpha = \sum_{D \in \KD(\alpha)} \bm{x}^{\wt(D)}
\]
\end{thm}
\begin{rem} \label{natual_map}
There is a natural identification 
between $\KD(\alpha)$ and $\RSSYT(\alpha)$ which yields $\KD(\alpha) = \RSSYT(\alpha)$. 
Take $T\in \RSSYT(\alpha)$. 
By our convention in the previous subsection,
$T$ is viewed as a diagram pair $(L, \emptyset)$. This result is well-known
to experts. For example, it follows from work done in~\cite{A}.    
For completeness, 
we will recover this result
in section \ref{section.describe KKD RSVT}.
\end{rem}

\begin{exa} \label{eg_021} Let $\alpha = (0,2,1)$, then 
\begin{align*}
\key(\alpha) &= 
\begin{ytableau}
3 & 2\cr
2 \cr
\end{ytableau}\,,\quad D(\alpha) = 
\raisebox{0.5cm}{
\begin{ytableau}
\cdot \cr
\cdot & \cdot \cr
\none
\end{ytableau}}\,, \text{ thus we obtain:}\\
\KD(\alpha) &= \Bigg\{
\raisebox{0.5cm}{
\begin{ytableau}
\cdot \cr
\cdot & \cdot \cr
\none
\end{ytableau}\,,
\quad
\begin{ytableau}
\none \cr
\cdot & \cdot \cr
\cdot
\end{ytableau}\,,
\quad
\begin{ytableau}
\cdot \cr
\cdot \cr
\none & \cdot 
\end{ytableau}\,,
\quad
\begin{ytableau}
\none \cr
\cdot \cr
\cdot & \cdot
\end{ytableau}\,,
\quad
\begin{ytableau}
\cdot \cr
\none \cr
\cdot & \cdot
\end{ytableau}}\Bigg\}\,, \text{ and}\\
\RSSYT(\alpha) &= \Bigg\{
\begin{ytableau}
3  & 2 \cr
2
\end{ytableau}\,,
\quad  
\begin{ytableau}
2  & 2 \cr
1
\end{ytableau}\,,
\quad 
\begin{ytableau}
3  & 1 \cr
2
\end{ytableau}\,,
\quad 
\begin{ytableau}
2  & 1 \cr
1
\end{ytableau}\,,
\quad 
\begin{ytableau}
3  & 1 \cr
1
\end{ytableau}\Bigg\}\,.
\quad 
\end{align*}
If we view each $\RSSYT$ as a diagram pair,
it is clear that $\RSSYT(\alpha) = \KD(\alpha)$.
\end{exa}

\smallskip
Ross and Yong~\cite[Section~1.2]{RossY} generalized Kohnert moves. We state their construction below.

A \definition{$K$-Kohnert move} is an operation on a diagram pair.
It selects the rightmost cell in a row. The selected cell cannot be a ghost cell. Then move this cell downward to the first position available. It can jump over other Kohnert cells, but cannot jump over any ghost cells. After the move, it may or may not leave a ghost cell at the original position. When a $K$-Kohnert move leaves a ghost cell, we also refer it as a \definition{ghost move}. Let $\KKD(\alpha)$ be the closure of $\{D_\alpha\}$ under all possible $K$-Kohnert moves.
We make the following observations.

\begin{rem}
Let $\alpha$ be a weak composition.
We have 
\begin{enumerate}
\item[$\bullet$]
$\KD(\alpha) \subseteq \KKD(\alpha)$.
\item[$\bullet$] If $(K, G) \in
\KKD(\alpha)$,
then $(K, \emptyset) \in \KD(\alpha)$.
\end{enumerate}
\end{rem}

\begin{rem}
Usually, an element of $\KD(\alpha)$ is viewed as a diagram.
We defined $\KD(\alpha)$ as a set of diagram pairs
so we can work with $\KD(\alpha)$ and $\KKD(\alpha)$
using the same technique.
In particular, with our convention,
$\KD(\alpha)$ is viewed as a subset of $\KKD(\alpha)$.
\end{rem}

Ross and Yong~\cite{RossY}
conjectured a formula for Lascoux 
polynomials involving $K$-Kohnert diagrams.

\begin{conj}~\cite{RossY} The Lascoux polynomials indexed by $\alpha$, is given by
\[
\fL^{(\beta)}_\alpha = \sum_{D \in \KKD(\alpha)}\beta^{\ex(D)}\bm{x}^{\wt(D)}\,.
\]
\end{conj}

We prove this conjecture by establishing
bijections between $\KKD(\alpha)$ and $\RSVT(\alpha)$
that preserve $\wt(\cdot)$ and $\ex(\cdot)$.
Moreover, when restricted to $\KD(\alpha)\subseteq \KKD(\alpha)$
and $\RSSYT(\alpha) \subseteq \RSVT(\alpha)$,
our bijections restrict to the identity maps.
We will describe our bijections 
in the next two subsections.

\begin{exa} \label{eg_021k}Continue Example~\ref{eg_021} for $\alpha = (0,2,1)$, we get
\begin{align*}
\KKD(\alpha) &= \KD(\alpha) \bigcup \Bigg\{
\raisebox{0.5cm}{
\begin{ytableau}
\X \cr
\cdot & \cdot \cr
\cdot
\end{ytableau}\,,
\quad
\begin{ytableau}
\cdot \cr
\cdot & \X \cr
\none & \cdot
\end{ytableau}\,,
\quad
\begin{ytableau}
\none \cr
\cdot & \X\cr
\cdot & \cdot 
\end{ytableau}\,,
\quad
\begin{ytableau}
\X \cr
\cdot \cr
\cdot & \cdot
\end{ytableau}\,,
\quad
\begin{ytableau}
\cdot \cr
\X \cr
\cdot & \cdot
\end{ytableau}\,,
\quad
\begin{ytableau}
\X \cr
\cdot & \X \cr
\cdot & \cdot
\end{ytableau}}\Bigg\}\,, \text{ and}\\
\RSVT(\alpha) &= \RSSYT(\alpha) \bigcup\Bigg\{
\begin{ytableau}
3  & 2 \cr
21
\end{ytableau}\,,
\quad  
\begin{ytableau}
3  & 21 \cr
2
\end{ytableau}\,,
\quad 
\begin{ytableau}
2  & 21 \cr
1
\end{ytableau}\,,
\quad 
\begin{ytableau}
32  & 1 \cr
1
\end{ytableau}\,,
\quad 
\begin{ytableau}
3  & 1 \cr
21
\end{ytableau}\,,
\quad 
\begin{ytableau}
3  & 21 \cr
21
\end{ytableau}\Bigg\}\,.
\end{align*}
Viewing the $6$ elements in $\RSVT(\alpha)$ with at least one extra number as diagram pairs, we obtain the following. Note they are different from the elements in $\KKD(\alpha)$ with at least one ghost cell.
\[
\Bigg\{
\raisebox{0.5cm}{
\begin{ytableau}
\cdot \cr
\cdot & \cdot \cr
\X
\end{ytableau}\,,
\quad
\begin{ytableau}
\cdot \cr
\cdot & \cdot \cr
\none & \X
\end{ytableau}\,,
\quad
\begin{ytableau}
\none \cr
\cdot & \cdot \cr
\cdot & \X 
\end{ytableau}\,,
\quad
\begin{ytableau}
\cdot \cr
\X \cr
\cdot & \cdot
\end{ytableau}\,,
\quad
\begin{ytableau}
\cdot \cr
\cdot \cr
\X & \cdot
\end{ytableau}\,,
\quad
\begin{ytableau}
\cdot \cr
\cdot & \cdot \cr
\X & \X
\end{ytableau}}\Bigg\}
\]

Thus, Conjecture~\ref{Main conjecture} is correct
when $\alpha = (0, 2, 1)$ since
$$
\sum_{D \in \KKD(\alpha)}\beta^{\ex(D)}\bm{x}^{\wt(D)}
= \sum_{T \in \RSVT(\alpha)}\beta^{\ex(T)}\bm{x}^{\wt(T)}.
$$
\end{exa}

\subsection{Kohnert Tableaux}
\label{Kohnert Tableaux}
Assaf and Searles introduced Kohnert tableaux
in~\cite{A}.
We will use Kohnert tableaux
to prove the correctness of our bijections.

\begin{defn}~\cite[Definition 2.3]{A}
\label{def: KT}
Let $\alpha$ be a weak composition. A {\em \definition{Kohnert tableau} with content $\alpha$}
is a Young diagram filled by numbers
such that:  
\begin{enumerate}
\item Column $c$ of the tableau 
consists of numbers $\{i: \alpha_i \geqslant c\}$,
with each number appearing exactly once.
\item If a number $i$ appears in row $r$,
then $i \geqslant r$.
\item If a number $i$ appears in $(c,r)$ 
and $(c + 1, r')$,
then $r \geqslant r'$.
\item Let $i, j$  appear in column $c$ 
with $j > i$ and $j$ is lower than $i$.
Then there is an $i$ in column $c+1$
that is strictly above the $j$ in column $c$.
\end{enumerate}

Let $\KT(\alpha)$ be the set of all
Kohnert tableaux with content $\alpha$.
\end{defn}

Assaf and Searles constructed bijections 
between $\KT(\alpha)$ and $\KD(\alpha)$ in~\cite{A}.
For each \linebreak $T \in \KT(\alpha)$,
we may ignore its numbers and view each cell
as a Kohnert cell.
By~\cite[Lemma 2.4]{A},
the resulting diagram pair is in $\KD(\alpha)$.

The inverse of the map above is called 
{\em Kohnert Labeling with respect to $\alpha$},
denoted as $\Label_\alpha(\cdot)$.
We may describe it as the following algorithm
on certain diagram pairs. 

Let $D$ be an arbitrary diagram pair
such that column $c$ of $D$
has $|\{i: \alpha_i \geqslant c\}|$ 
Kohnert cells and no ghost cells.
Initialize sets $S_1, S_2, \dots$
as $S_c = \{i: \alpha_i \geqslant c\}$.
Iterate through boxes of $D$ from right to left,
and from bottom to top within each column.
For the box $(c,r)$, 
find the smallest $i \in S_c$ 
such that $i$ does not appear at $(c+1, r')$
for all $r' > r$.
We remove $i$ from $S_c$ and fill $i$
in $(c,r)$ of $D$.
If no such $i$ exists
or $i < r$,
terminate the algorithm.
After all boxes are filled, 
output the final tableau.

By~\cite[Lemma 2.6, Lemma 2.7, Theorem 2.8]{A},
the labeling algorithm on $D$ produces an output
if and only if $D \in \KD(\alpha)$.
Moreover, if we restrict the algorithm
on $\KD(\alpha)$, 
then this is a bijection from $\KD(\alpha)$
to $\KT(\alpha)$ whose inverse is described above.

\begin{exa} Let $\alpha = (0,2,1)$, we have 
\[
\KT(\alpha) = \Bigg\{
\raisebox{0.5cm}{
\begin{ytableau}
3 \cr
2 & 2 \cr
\none
\end{ytableau}\,,
\quad
\begin{ytableau}
\none \cr
2 & 2 \cr
3
\end{ytableau}\,,
\quad
\begin{ytableau}
3 \cr
2 \cr
\none & 2 
\end{ytableau}\,,
\quad
\begin{ytableau}
\none \cr
3 \cr
2 & 2
\end{ytableau}\,,
\quad
\begin{ytableau}
3 \cr
\none \cr
2 & 2
\end{ytableau}}\Bigg\}\,,
\]
where the relative order in the sets corresponds to $\KD(\alpha)$ from Example~\ref{eg_021} under the above labeling algorithm.
\end{exa}

\section{Describing the maps}
\label{section.maps}
For each composition $\alpha$, 
we have introduced two sets of diagram pairs: 
$\KKD(\alpha)$ and $\RSVT(\alpha)$.
We will define two maps: $\Psi_\alpha$ on $\KKD(\alpha)$ 
and $\Phi_\alpha$ on $\RSVT(\alpha)$.
In section~\ref{section.recursive},
we will show the image of $\Psi_\alpha$ (resp. $\Phi_\alpha$)
lies in $\RSVT(\alpha)$ (resp. $\KKD(\alpha)$).

\subsection{An informal description of $\Psi_\alpha$}\label{prelim.kkd.to.rsvt}
In this section, 
we describe a map $\Psi_\alpha$ from $\KKD(\alpha)$
to the set of all diagram pairs. 
First, we describe an operator on $\KD(\alpha)$.
Let $G$ be a diagram.
Then $\sharp_G(\cdot)$ acts on $\KD(\alpha)$ in the following way:
Take $D \in \KD(\alpha)$.
Iterate through cells of $G$ from right to left.
Within each column, go from bottom to top.
For $(c, r) \in G$,
search for the largest $r' \leqslant r$ such that
$(c, r')$ is a Kohnert cell in $D$.
Moreover, if we raise the cell $(c,r')$ to $(c,r)$,
the resulting diagram is still in $\KD(\alpha)$.
After finding such $r'$, 
we move cell $(c,r')$ to $(c,r)$.
After iterating over all cells in $G$, 
we denote the final Kohnert diagram by $\sharp_G(D)$.
If we cannot find such an $r'$ during an iteration,
then $\sharp_G(D)$ is undefined. 

\begin{exa} Let $D$ be the fourth Kohnert diagram in Example~\ref{eg_021}. 
Let $G$ be the diagram $\{ (1,3), (2,2)\}$.
We may compute $\sharp_G(D)$ as follows. 
We label $(c,r)$ and $(c,r')$ involved in each step above and below the arrows.
\begin{align*}
\begin{ytableau}
\none \cr
\cdot \cr
\cdot & \cdot
\end{ytableau}\;&\xrightarrow[(2,1)]{(2,2)}\;
\begin{ytableau}
\none \cr
\cdot & \cdot \cr
\cdot
\end{ytableau}\;\xrightarrow[(1,1)]{(1,3)}\;
\begin{ytableau}
\cdot \cr
\cdot & \cdot \cr
\none
\end{ytableau}
\end{align*}
\end{exa}
Now, we may describe the map 
$\Psi_\alpha$. 
Take $D = (K, G) \in \KKD(\alpha)$.
Compute $ (L, \emptyset) = \sharp_G((K,\emptyset)) \in \KD(\alpha)$.
Then $\Psi_\alpha(D) := (L, (K \sqcup G) - L)$.

\begin{exa} Let $D = (K, G)$ be the last diagram pair
in Example~\ref{eg_021k}.
The previous example shows $ \sharp_G(K) = (\{ (1,2), (1,3), (2,2)\}, \emptyset)$.
Thus, $D$ is sent to the diagram pair 
$(\{(1,2), (1,3), (2,2)\},$ $ \{(1,1), (2,1)\})$.
Notice that this diagram pair
corresponds to the following $\RSVT$:
\begin{align*} 
\raisebox{0.43cm}
{\begin{ytableau}
\cdot \cr
\cdot & \cdot \cr
\X & \X
\end{ytableau}}
=
\raisebox{0.43cm}{
\begin{ytableau}
3  & 21 \cr
21
\end{ytableau}}.
\end{align*}
\end{exa}

We say this is an informal description of $\Psi_\alpha$
since the map is not obviously well-defined. 
It seems possible that $\sharp_G((K,\emptyset))$ is undefined. 
In section~\ref{section.recursive},
we will provide an alternative description 
of the map $\Psi_\alpha$
and check the following.
\begin{lemma}
\label{L: Psi well-defined}
The map $\Psi_\alpha$ is a well-defined map
from $\KKD(\alpha)$ to $\RSVT(\alpha)$.
\end{lemma}

\subsection{An informal description of $\Phi_\alpha$}
\label{prelim.rsvt.to.kkd}
The map $\Phi_\alpha$
can be described similarly on $\RSVT(\alpha)$.
First, we need an analogue 
of the $\sharp_G(\cdot)$ operator.
Let $E$ be a diagram.
Then $\flat_E(\cdot)$ acts on $\KD(\alpha)$ in the following way:
Take $D \in \KD(\alpha)$.
Iterate through cells of $E$ from left to right.
Within each column, go from top to bottom.
For $(c, r) \in G$,
search for the smallest $r' \geqslant r$ such that
$(c, r')$ is a Kohnert cell in $D$.
Moreover, if we drop the cell $(c,r')$ to $(c,r)$,
the resulting diagram is still in $\KD(\alpha)$.
After finding such $r'$, 
we move cell $(c,r')$ to $(c,r)$.
After iterating over all cells in $G$, 
we denote the final Kohnert diagram by $\flat_G(D)$.
If we cannot find such an $r'$ during an iteration,
then $\flat_G(D)$ is undefined.

\begin{exa} Let $D$ be the first Kohnert diagram in Example~\ref{eg_021}.
Let $E$ be the diagram $\{ (1,1), (2,1) \}$.
We may compute $\flat_E(D)$ as follows. 
We label $(c,r)$ and $(c,r')$ involved in each step above and below the arrows.
\begin{align*}
\begin{ytableau}
\cdot \cr
\cdot & \cdot \cr
\none
\end{ytableau}\;&\xrightarrow[(1,3)]{(1,1)}\;
\begin{ytableau}
\none \cr
\cdot & \cdot \cr
\cdot
\end{ytableau}\;\xrightarrow[(2,2)]{(2,1)}\;
\begin{ytableau}
\none \cr
\cdot \cr
\cdot & \cdot
\end{ytableau}
\end{align*}
\end{exa}
Now, we may describe the map 
$\Phi_\alpha$. 
Take $T \in \RSVT(\alpha)$
and we may write $T$ as a diagram pair $(L, E)$.
Compute $(K,\emptyset) = \flat_E((L, \emptyset))$.
Then $\Phi_\alpha(T) := (K, (L \sqcup E) - K)$.

\begin{exa} Let $T$ be the last $\RSVT$ 
in Example~\ref{eg_021k}.
We may write $T = (L, E)$ as \\
$(\{(1,2), (1,3), (2,2)\},$ $\{(1,1), (2,1)\})$.
The previous example computes 
$(K,\emptyset) = \flat_E((L, \emptyset))$.
Thus, $(L \sqcup E) - K = \{(1,3), (2,2)\}$ and
$T$ is sent to

\begin{align*}
\begin{ytableau}
\X \cr
\cdot & \X \cr
\cdot & \cdot
\end{ytableau}
\end{align*}
Notice that this is an element of $\KKD(\alpha)$.
\end{exa}

Again, $\Phi_\alpha$ is not obviously well-defined.
In section~\ref{section.recursive},
we will provide an alternative description 
of the map $\Phi_\alpha$
and check the following.
\begin{lemma}
\label{L: Phi well-defined}
The map $\Phi_\alpha$ is a well-defined map
from $\RSVT(\alpha)$ to $\KKD(\alpha)$.
\end{lemma}

Now we restate our main result.
The proof is also in section~\ref{section.recursive}.

\begin{thm}
The maps $\Psi_\alpha: \KKD(\alpha) \to \RSVT(\alpha)$ 
and $\Phi_\alpha: \RSVT(\alpha) \to \KKD(\alpha)$
are mutually inverses of each other.
Moreover, they preserve $\wt(\cdot)$ and $\ex(\cdot)$.
\end{thm}

\section{Bruhat Order on Weak Compositions}
\label{section.Bruhat}
Partial order on weak compositions has been studied in~\cite{A,A19,FG,FGPS}. In this section, we give a definition via the key tableau associated to the weak composition. In subsection~\ref{subsection.bruhat}, we also study $m(\alpha,S)$ (resp. $M(\alpha,S)$) which is the unique minimum (resp. maximum) weak composition in a certain set of weak compositions. In subsection~\ref{subsection.left.swap}, we used properties of $M(\alpha,S)$ to show that our Bruhat order is equivalent to the left swap order defined in~\cite{A,A19}, which implies that the Bruhat order is equivalent to the inclusion order on $\KD(\alpha),\KKD(\alpha), \RSSYT(\alpha)$ and $\RSVT(\alpha)$.

\subsection{Bruhat order}\label{subsection.bruhat}

We may define a partial order on all weak compositions.
\begin{defn}
Let $\alpha, \gamma$ be two weak compositions. 
We define $\alpha \leqslant \gamma$ if 
$\key(\alpha)$ and $\key(\gamma)$ have the same shape
and $\key(\alpha) \leqslant \key(\gamma)$ entry-wise.
This order is called the \definition{Bruhat order}.  
\end{defn}

\begin{defn}
For $S \subseteq [n]$, 
let $\mathbbm{1}_S$ be the weak composition
whose $i^{th}$ entry is 1 if $i \in S$
and 0 otherwise.
For a weak composition $\alpha$,
the \definition{support} of $\alpha$
is the set
$\{i |\, \alpha_i > 0\}$, denoted as $\supp(\alpha)$.
\end{defn}

Weak compositions with only 0s and 1s are in
natural bijection with subsets of $[n]$.
The bijections are $S \mapsto \mathbbm{1}_S$
and $\alpha \mapsto \supp(\alpha)$.
With the Bruhat order above, 
we may define a partial order
on subsets of $[n]$.
\begin{defn}
For $S, S' \subseteq [n]$, 
we say $S \leqslant S'$ if 
$\mathbbm{1}_S \leqslant \mathbbm{1}_{S'}$.
\end{defn}

We have other alternative descriptions 
of this order.

\begin{lemma}\label{lem:subset_order}
Take $S, S' \subseteq [n]$.
The following are equivalent:
\begin{enumerate}
\item $S \leqslant S'$
\item $|S| = |S'|$ and for each $j \in [|S|]$,
the $j^{th}$ largest number in $S$ is at most
the $j^{th}$ largest number in $S'$.
\item $|S| = |S'|$ and for each $s \in S$,
$|[s, n] \cap S| \leqslant |[s, n] \cap S'|$.
\end{enumerate}
\end{lemma}
\begin{proof}
We first prove statement 1 and 2 are equivalent.
By definition $S \leqslant S'$ if and only if 
$\key(\mathbbm{1}_S) \leqslant \key(\mathbbm{1}_{S'})$.
The number on row $j$ column 1 of $\key(\mathbbm{1}_S)$ 
(resp. $\key(\mathbbm{1}_{S'})$)
is the $j^{th}$ largest number in $S$ (resp. $S'$). 
Thus, $\key(\mathbbm{1}_S) \leqslant \key(\mathbbm{1}_{S'})$
is equivalent to statement 2.

Next, we show the statement 2 and 3 are equivalent.
Assume the statement 2 is true. 
Let $s \in S$ be the $j^{th}$ largest number in $S$.
Then $|[s, n] \cap S| = j$.
Since the $j^{th}$ largest number in $S'$ is at least $s$,
$|[s, n] \cap S'| \geqslant j$.
Now assume statement 3 is true.
Let $s$ be the 
$j^{th}$ largest number in $S$.
We know there are at least $j$ numbers 
in $[s, n] \cap S'$,
so the $j^{th}$ largest number in $S'$
is at least $s$.
\end{proof}

Take $S \subseteq [n]$ and a weak composition $\alpha$.
Consider the set $\{ \gamma: \gamma \geqslant \alpha, \supp(\gamma) \subseteq S\}$.
Let $m(\alpha, S)$ be the unique minimum element in the set, if it exists.
Later, we will show $m(\alpha,S)$ exists as long as
the set is non-empty.
First, we introduce an algorithm 
to compute $m(\alpha, S)$ or assert
it does not exist.
Initialize $list$ to be an empty list
and initialize $\sigma$ to be
the weak composition with all 0s.
Iterate over $i  = 1, \dots, n$.
Perform the following two processes in each iteration:
\begin{itemize}
\item
(Adding process):
If $\alpha_i > 0$, 
then add $\alpha_i$ to $list$.
\item
(Removing process):
If $i \in S$ and $list$ is non-empty,
then remove $\max(list)$ from $list$
and assign it to $\sigma_i$.
\end{itemize}
After all iterations,
if $list$ is empty,
then $m(\alpha, S)$ is $\sigma$.
Otherwise, such $m(\alpha, S)$ does not exist.
 
\begin{exa}
Let $n = 7$, $\alpha = (1,3,0,2,0,0,2)$ and $S = \{3,4,5,6,7\}$.
Then we trace $\sigma$ and $list$ during the algorithm:
\begin{enumerate}
\item[$\bullet$] Before iteration 1: $\sigma = (0,0,0,0,0,0,0)$; $list$ is empty. 
\item[$\bullet$] After iteration 1:  $\sigma = (0,0,0,0,0,0,0)$; $list$ contains $1$.
\item[$\bullet$] After iteration 2:  $\sigma = (0,0,0,0,0,0,0)$; $list$ contains $1,3$. 
\item[$\bullet$] After iteration 3:  $\sigma = (0,0,3,0,0,0,0)$; $list$ contains $1$. 
\item[$\bullet$] After iteration 4:  $\sigma = (0,0,3,2,0,0,0)$; $list$ contains $1$. 
\item[$\bullet$] After iteration 5:  $\sigma = (0,0,3,2,1,0,0)$; $list$ is empty. 
\item[$\bullet$] After iteration 6:  $\sigma = (0,0,3,2,1,0,0)$; $list$ is empty. 
\item[$\bullet$] After iteration 7:  $\sigma = (0,0,3,2,1,0,2)$; $list$ is empty. 
\end{enumerate}
Since the list is empty after the iterations, 
the algorithm outputs $m(\alpha, S) = (0,0,3,2,1,0,2)$.
\end{exa}

We take steps to show this algorithm is correct.
We start with the following observation,
which connects this algorithm with the $\lhd$
operator in Definition~\ref{def:left_key}.

\begin{lemma}
\label{lem:algorithm_lhd}
Assume the algorithm outputs $m(\alpha, S) = \sigma$.
Let $T_c$ (resp. $A_c$) be the set consisting of numbers
in column $c$ of $\key(\sigma)$ (resp. $\key(\alpha)$).
Then $T_c = S \lhd A_c$.
\end{lemma}

\begin{proof}
We know $i \in A_c$ if and only if 
during the adding process of the iteration $i$,
a number at least $c$ is added to $list$.
Similarly, $i \in T_c$ if and only if
during the removing process of the iteration $i$,
a number at least $c$ is removed from $list$.

Assume $T_c = \{t_1<t_2<\dots <t_s\}$ 
and $A_c =\{a_1<a_2<\dots < a_s\}$.
During the adding process of iteration $a_1$,
the algorithm puts $\alpha_{a_1}$ into $list$.
This is the first time that $list$
gains a number at least $c$.
Thus, $t_1 \geqslant a_1$.
Moreover, assume there exists $t \in S$
such that $a_1 \leqslant t < t_1$.
During the removing process of iteration $t$,
$list$ has a number at least $c$.
The algorithm will remove a number at least $c$
from $list$,
contradicting to $t \notin T_c$.
Thus, $t_1$ is the smallest in $S$
with $t_1 \geqslant a_1$.

Now consider $t_j$ with $j > 1$.
During the removing process of $t_j$, 
we remove a number at least $c$ for the $j^{th}$ time.
Thus, we have added at least $j$ such numbers to $list$,
so $a_j \leqslant t_j$.
Now assume there is $t < t_j$ 
such that $t \in S$, $t \geqslant a_j$, and $t > t_{j - 1}$.
During the removing process of iteration $t$,
there is a number at least $c$ in $list$,
so such a number will be removed. 
We have a contradiction since $t \notin T_c$.
Thus, $t_j$ is the smallest in $S$
such that $t_j \geqslant a_j$ and $t_j > t_{j - 1}$.
\end{proof}

Next, we investigate the condition for the algorithm
to assert $m(\alpha, S)$ does not exist. 

\begin{lemma}
The following are equivalent:
\begin{enumerate}
\item $|S| \geqslant |\supp(\alpha)|$. 
In addition, if we let $S' \subseteq S$ 
consists of the largest numbers in $S$
with $|S'| = |\supp(\alpha)|$,
then $S' \geqslant \supp(\alpha)$.
\item The algorithm asserts $m(\alpha, S)$ exists. 
(i.e. The $list$ is empty when the algorithm ends.)
\item There is a weak composition $\gamma$ 
such that $\supp(\gamma) \subseteq S$
and $\gamma \geqslant \alpha$.
\end{enumerate}
\end{lemma}
\begin{proof}

First, assume $|S| < |\supp(\alpha)|$.
We check the last two statements do not hold.
\begin{enumerate}
\setcounter{enumi}{1}
\item Consider the $i^{th}$ iteration.
First, 
the size of $list$ is increased by one 
if $i \in \supp(\alpha)$.
Next, the size of $list$ is fixed 
or decreased by one if $i \in S$.
Throughout this algorithm, 
$list$ gains $|\supp(\alpha)|$ numbers 
and loses at most $|S|$ numbers.
It is not empty when the algorithm ends. 
\item 
Assume such $\gamma$ exists.
We know $|\supp(\gamma)| = |\supp(\alpha)| > |S|$,
contradicting to $\supp(\gamma) \subseteq S$.
\end{enumerate}

Now assume $|S| \geqslant |\supp(\alpha)|$
and define $S'$ as above. 
Suppose $S'\ngeqslant \supp(\alpha)$. 
We check the last two statements in the lemma 
do not hold.
\begin{enumerate}
\setcounter{enumi}{1}
\item By Lemma~\ref{lem:subset_order}, there exists $j \in \supp(\alpha)$ such that 
$$
|[j,n]\cap \supp(\alpha)|>|[j,n]\cap S'| 
= |[j,n]\cap S|.
$$
Between the $j^{th}$ iteration and the last iteration inclusively,
$list$ gains $|[j,n]\cap \supp(\alpha)|$ numbers and loses
at most $|[j,n]\cap S|$ numbers. 
It is not empty when the algorithm ends. 
\item 
Assume such $\gamma$ exists.
By $\supp(\gamma) \subseteq S$,
we have $S' \geqslant \supp(\gamma) \geqslant \supp(\alpha)$.
Contradiction.
\end{enumerate}

Finally, assume $S' \geqslant \supp(\alpha)$.
We check the last two statements in the lemma
are true. 
\begin{enumerate}
\setcounter{enumi}{1}
\item Assume when the algorithm ends,
$list$ is not empty.
Find the largest $j$ such that
$list$ is not empty since the $j^{th}$ 
iteration.
First, we know $list$ is empty 
right before the $j^{th}$ iteration.
Second, we know a number is added to $list$
during the $j^{th}$ iteration, 
so $j \in \supp(\alpha)$.
By Lemma~\ref{lem:subset_order}, 
$|[j,n] \cap \supp(\alpha)| 
\leqslant |[j, n] \cap S'| = |[j, n] \cap S|$.
Between the $j^{th}$ iteration 
and the last iteration inclusively,
$list$ gains $|[j,n]\cap \supp(\alpha)|$ numbers.
Since $list$ is not empty since iteration $j$,
$list$ loses $|[j, n] \cap S|$ numbers.
Thus, $list$ is empty after the last iteration. 
\item 
From the previous statement,
we know the algorithm will produce a weak composition $\sigma$.
Just need to check $\supp(\sigma) \in S$ 
and $\sigma \geqslant \alpha$.
It is apparent that $\supp(\sigma)\subseteq S$. Since all positive numbers in $\alpha$ are assigned into $\sigma$, $\key(\sigma)$ and $\key(\alpha)$ are of the same shape. Next we show that $\sigma \geqslant \alpha$ by comparing $\key(\sigma)$ and $\key(\alpha)$ column-by-column: 
Let $T_c$ (resp. $A_c$) be the set consisting of numbers
in column $c$ of $\key(\sigma)$ (resp. $\key(\alpha)$).
By Lemma~\ref{lem:algorithm_lhd}, $T_c = S \lhd A_c$.
Thus, $T_c \geqslant A_c$.
\end{enumerate}
\end{proof}

Now we can prove the correctness of our algorithm.

\begin{lemma}\label{m_algorithm}
The algorithm correctly computes $m(\alpha, S)$.

In other words, consider the set 
$\{ \gamma: \gamma \geqslant \alpha, 
\supp(\gamma) \subseteq S\}$.
\begin{enumerate}
\item[$\bullet$] If $list$ is empty after the iterations,
then the output $\sigma$ is the unique minimum in the set. 
\item[$\bullet$] Otherwise,
the set is empty. 
\end{enumerate}

Moreover, the second case happens only when
the set is empty. 
\end{lemma}
\begin{proof}
If $list$ is not empty when the algorithm ends,
$\{ \gamma: \gamma \geqslant \alpha, 
\supp(\gamma) \subseteq S\} = \emptyset$
by the previous lemma.

Now assume $list$ is empty when the algorithm ends.
Then the set is non-empty: 
By the proof of the previous lemma, 
the output $\sigma$ is in this set. 
We check $\sigma$ is the least element.
Assume there is a $\gamma$ in the set with $\gamma \ngeqslant \sigma$.
Let $T_c'$ consists of numbers in column $c$ of $\key(\gamma)$.
Define $T_c$ and $A_c$ similarly for $\key(\sigma)$ and $\key(\alpha)$ respectively. 
Then we can find $c$ such that $T_c' \ngeqslant T_c$.
Let $t_j'$ be the $j^{th}$ smallest number in $T_c'$.
Define $t_j$ and $a_j$ similarly for $T_c$ and $A_c$.
Then we can find smallest $j$ such that $t_j' < t_j$.
Notice that $t_j' \geqslant a_j'$ and $t_j' \in S$.
Moreover, $t_j' > t_{j-1}' \geqslant t_{j-1}$ if $j > 1$.
We have a contradiction to the fact $T_c = S \lhd A_c$ 
from Lemma~\ref{lem:algorithm_lhd}.
\end{proof}

\begin{cor}
\label{C: m exists}
Let $\alpha$ be a weak composition
and $S \subseteq [n]$.
$m(\alpha, S)$ exists if and only if
\begin{enumerate}
\item[$\bullet$] $|S| \geqslant |\supp(\alpha)|$, and
\item[$\bullet$] Let $S' \subseteq S$ 
consists of the largest numbers in $S$
with $|S'| = |\supp(\alpha)|$.
Then $S' \geqslant \supp(\alpha)$.
\end{enumerate}
\end{cor}
\begin{proof}
Follows from the previous two lemmas.
\end{proof}

Analogously, we may also look at the set 
$\{ \gamma: \gamma \leqslant \alpha, \supp(\gamma) \subseteq S\}$.
Similarly, if it is non-empty, it will contain
a unique maximum element. 
Let $M(\alpha, S)$ be this element.
To compute it, 
we only need to slightly change our algorithm above:
Let $i$ goes from $n$ to 1, instead of $1$ to $n$.
Similar to Corollary~\ref{C: m exists},
we have the following for $M(\alpha, S)$.
\begin{cor}
\label{C: M exists}
Let $\alpha$ be a weak composition
and $S \subseteq [n]$.
$M(\alpha, S)$ exists if and only if
\begin{enumerate}
\item[$\bullet$] $|S| \geqslant |\supp(\alpha)|$, and
\item[$\bullet$] Let $S' \subseteq S$ 
consists of the smallest numbers in $S$
with $|S'| = |\supp(\alpha)|$.
Then $S' \leqslant \supp(\alpha)$.
\end{enumerate}
\end{cor}
\begin{proof}
The proof is similar to the proof of Corollary~\ref{C: m exists}.
\end{proof}

We end this subsection by a property
that connects $m(\gamma, S)$ and $M(\alpha, S)$.
Let $\gamma$ be a weak composition.
We use $\overline{\gamma}$ to denote
the weak composition obtained
by decreasing each positive entry of $\gamma$
by 1. 
\begin{lemma}
\label{L: m and M}
Let $\alpha, \gamma$ be two weak compositions.
Take $S \subseteq [n]$ 
with $|S| = |\supp(\alpha)|$.
Then the following are equivalent:
\begin{enumerate}
\item[$\bullet$] $m(\gamma, S)$ exists
and $\alpha \geqslant \mathbbm{1}_S + m(\gamma, S)$.
\item[$\bullet$] $M(\alpha, S)$ exists
and $\overline{M(\alpha, S)} \geqslant \gamma$.
\end{enumerate}
\end{lemma}
\begin{proof}
Assume the first statement is true. 
Notice $\supp(\mathbbm{1}_S + m(\gamma, S)) = S$,
so $M(\alpha, S)$ exists and 
$$M(\alpha, S) \geqslant \mathbbm{1}_S + m(\gamma, S).$$
Decrease each positive entry by 1 on both sides and get
$$\overline{M(\alpha, S)} \geqslant m(\gamma, S).$$
Then we get the second statement since 
$m(\gamma, S) \geqslant \gamma$.

Now assume the second statement is true. 
Notice $\supp(\overline{M(\alpha, S)}) \subseteq S$,
so $m(\gamma, S)$ exists and 
$$\overline{M(\alpha, S)} \geqslant m(\gamma, S).$$
By $|S| = |\supp(\alpha)|$, 
$\supp(M(\alpha, S)) = S$,
so $$M(\alpha, S) \geqslant \mathbbm{1}_S + m(\gamma, S).$$
Then we get the first statement since 
$\alpha \geqslant M(\alpha, S)$.
\end{proof}

\subsection{Left swap order}\label{subsection.left.swap}
Assaf and Searles~\cite{A} 
also defined a partial order 
on weak compositions 
called the \definition{left swap order}.
In this subsection, 
we introduce this order and show
that it is equivalent to the Bruhat order.

\begin{defn}\cite[Definition 2.3.4]{A,A19}\label{left_swap_inclusion}
A \definition{left swap} on a weak composition $\alpha$ exchanges two parts $\alpha_i<\alpha_j$ with $i<j$. The left swap order on weak compositions is the transitive closure of the relation $\gamma\preceq \alpha$ whenever $\gamma$ is a left swap of $\alpha$.
\end{defn}

When $\gamma$ is obtain from $\alpha$ by exchanging the $i^{th}$ and $j^{th}$ parts of $\alpha$, we write $\gamma = (i\: j)\alpha$.

\begin{prop}\cite[Prop. 2.3.9]{A19}\label{left_swap_inclusion}
Given weak compositions $\alpha,\gamma$, we have $\gamma \preceq \alpha$ if and only if the key diagram pair of $\gamma$ is in $\KD(\alpha)$.
\end{prop}

To show the equivalence between the left swap order
and the Bruhat order, 
we need a few lemmas.
We start with the following, 
which summarizes how $M(\alpha, S)$
is changed when we changes $S$ in a nice way.

\begin{lemma}
\label{L: Lift M}
Let $\alpha$ be a weak composition
and $S \subseteq [n]$ with $S \leqslant \supp(\alpha)$.
Take $g \notin S$ such that 
$|[g, n] \cap \supp(\alpha)| > |[g, n] \cap S|$.
Then there exists $s \in S$ such that $s < g$ and
$$M(\alpha, S')
= (s \:\: g) M(\alpha, S),$$
where $S' = (S \sqcup \{g\}) - \{s\}$.
In particular, 
$M(\alpha, S)$is a left swap of $M(\alpha, S')$.
\end{lemma}
\begin{proof}
Run the algorithm that computes
$M(\alpha, S)$.
After initialization, 
the algorithm iterates from $i = n$
to $i = 1$.
Right after the iteration with $i = g$, 
the algorithm has put 
$|[g, n] \cap \supp(\alpha)|$
numbers to $list$
and has removed at most
$|[g, n] \cap S|$ numbers from $list$.
Since $|[g, n] \cap \supp(\alpha)|
> |[g, n] \cap S|$,
the list is non-empty. 
Let $x$ be the largest number 
in the current list.
Then this $x$ will be picked sometime 
in the future.
Let $s \in S \cap [1,g)$ be the largest 
such that in the iteration $i = s$, 
the algorithm assigns $x$ to $\sigma_i$.
We may run the algorithm to compute 
$M(\alpha, (S \sqcup \{g\}) - \{s\})$.
It behaves the same as on 
$M(\alpha, S)$,
except it assigns $x$ to $\sigma_g$
and keeps $\sigma_s = 0$.
Thus, $M(\alpha, (S \sqcup \{g\}) - \{s\})
= (g \; s) M(\alpha, S)$.
\end{proof}

We know from definition that 
$M(\alpha, S) \leqslant \alpha$.
The next lemma will describe their relationship
in the left swap order.

\begin{lemma}
\label{L: M and left swap}
Let $\alpha$ be a weak composition.
Let $S$ be a set such that $S \leqslant \supp(\alpha)$.
Then $M(\alpha, S) \preceq \alpha$.
\end{lemma}
\begin{proof}
Find the smallest $g \geqslant 1$
such that $[g, n] \cap S = [g, n] \cap \supp(\alpha)$.
Prove this lemma by induction on $g$.
For the base case, we assume $g = 1$.
Then $\supp(\alpha) = S$
which implies $M(\alpha, S) = \alpha$.

Next, assume $g > 1$.
Since $[g - 1, n] \cap S \neq [g - 1, n] \cap \supp(\alpha)$
and $S \leqslant \supp(\alpha)$,
we know $g-1 \in \supp(\alpha) - S$.
Thus, $|[g - 1, n] \cap S| < |[g - 1, n] \cap \supp(\alpha)|$.
By Lemma~\ref{L: Lift M}, 
there exists $s \in S$ such that 
if we let $S' = (S - \{s\}) \sqcup \{g-1\}$,
we have $M(\alpha, S) \preceq M(\alpha, S')$.
Notice that $[g - 1, n] \cap S' = [g - 1, n] \cap \supp(\alpha)$.
We may apply our inductive hypothesis and get
$M(\alpha, S) \preceq M(\alpha, S') \preceq \alpha$.
\end{proof}

Finally, we need the following intuitive lemma,
which says both partial orders are preserved 
by the operator $\alpha \mapsto \overline{\alpha}$. 
\begin{lemma}
\label{L: bar order}
Given weak compositions $\alpha$ and $\gamma$
with $\supp(\alpha) = \supp(\gamma)$.
Then we have
\begin{itemize}
\item $\gamma \leqslant \alpha$
if and only if $\overline{\gamma} \leqslant \overline{\alpha}$.
\item $\gamma \preceq \alpha$
if and only if $\overline{\gamma} \preceq \overline{\alpha}$.
\end{itemize}
\end{lemma}
\begin{proof}
Immediate from definitions.
\end{proof}

Now we are ready to prove the equivalence of these two partial orders.

\begin{prop} \label{left_swap_equiv}
Given weak compositions $\alpha$ and $\gamma$, we have $\gamma \preceq \alpha$ if and only if $\gamma \leqslant \alpha$.
\end{prop}

\begin{proof}
First we show if $\gamma \preceq \alpha$, then $\gamma \leqslant \alpha$. 
It suffices to show when $\gamma$ is a left swap of $\alpha$. 
Say $\gamma = (i \: j) \alpha$ where $i < j$ and $\alpha_i < \alpha_j$. 
Let $T_c$ (resp. $T'_c$) consists of 
numbers in column $c$ of $\key(\alpha)$ 
(resp. $\key(\gamma)$).
When $c \leqslant \alpha_i$ or $c>\alpha_j$, we have $T_c = T'_c$. 
When $\alpha_i< c \leqslant \alpha_j$, $T'_c$ is obtained from $T_c$ by replacing $j$ with $i$. Therefore $\key(\gamma) \leqslant \key(\alpha)$ and $\gamma \leqslant \alpha$.

Next we assume $\gamma \leqslant \alpha$ and 
show $\gamma \preceq \alpha$. 
We prove by induction on $\max(\alpha)$.
If $\max(\alpha) = 0$, 
then $\alpha, \gamma$ only contain 0s. 
Our claim is immediate. 
Now assume $\max(\alpha) > 0$.
We consider two cases.
\begin{itemize}
\item If $\supp(\gamma) = \supp(\alpha)$,
then $\overline{\gamma} \leqslant \overline{\alpha}$
by Lemma~\ref{L: bar order}.
By our inductive hypothesis, 
$\overline{\gamma} \preceq \overline{\alpha}$.
By Lemma~\ref{L: bar order} again,
$\gamma \preceq \alpha$.
\item Assume $\supp(\gamma) \neq \supp(\alpha)$.
Let $S = \supp(\gamma)$.
First, notice that $\gamma \leqslant M(\alpha, S)$
and these two weak compositions have the same support.
By the previous case, 
$\gamma \preceq M(\alpha, S)$.
It remains to check $M(\alpha, S) \preceq \alpha$,
which follows from
$S \leqslant \supp(\alpha)$ and
Lemma~\ref{L: M and left swap}.

\end{itemize}
\end{proof}

Consequently, we know several statements
are equivalent to $\gamma \leqslant \alpha$.

\begin{cor}
\label{C: KD RSVT order}
Given weak compositions $\alpha,\gamma$, the following are equivalent:
\begin{enumerate}
    \item $\gamma \leqslant \alpha$;
    \item $\KD(\gamma)\subseteq \KD(\alpha)$;
    \item $\KKD(\gamma)\subseteq \KKD(\alpha)$;
    \item $\RSSYT(\gamma)\subseteq \RSSYT(\alpha)$;
    \item $\RSVT(\gamma)\subseteq \RSVT(\alpha)$.
\end{enumerate}
\end{cor}
\begin{proof}
We show the following directions.
\begin{itemize}
    \item $(1)\Longleftrightarrow (2)$  This follows from Propositions~\ref{left_swap_inclusion} and \ref{left_swap_equiv}.
    \item $(1)\Longleftrightarrow (4)$ This is true by definition.
    \item $(1)\Longleftrightarrow (5)$ This is true by definition.
    \item $(2)\implies (3)$  This is true by definition.
    \item $(3)\implies (2)$  Since $\KKD(\gamma)\subseteq \KKD(\alpha)$ and $\KD(\gamma)\subseteq \KKD(\gamma)$, we have $\KD(\gamma)\subseteq \KKD(\alpha)$. Since there is no ghost cells in elements of $\KD(\gamma)$, we have $\KD(\gamma)\subseteq \KD(\alpha)$.
\end{itemize}
\end{proof}

\section{Recursive descriptions of $\KKD(\alpha)$ and $\RSVT(\alpha)$}
\label{section.describe KKD RSVT}
To check our maps $\Phi_\alpha$ and $\Psi_\alpha$ are well-defined,
we need to study their domains. 
In this section, we give necessary and sufficient criteria on when a diagram pair is in $\KKD(\alpha)$ (resp. $\RSVT(\alpha)$).
Our criteria will be recursive. 

For a diagram pair $D = (K, G)$,
we can send it to a triple $(K_1, G_1, d)$
where $K_1$, $G_1$ are disjoint 
subsets of $[n]$ and $d$ is a diagram pair.
They are defined as follows:
\begin{enumerate}
\item[$\bullet$] $K_1 := \{r : (1,r) \in K\}$
\item[$\bullet$] $G_1 := \{r : (1,r) \in G\}$
\item[$\bullet$] $d$ is the diagram pair 
with a Kohnert cell (resp. ghost cell) at $(c,r)$
if $D$ has a Kohnert cell 
(resp. ghost cell) at $(c + 1,r)$ with $c \geqslant 1$.
\end{enumerate}

The map $D \rightarrow (K_1, G_1, d)$ is invertible:
given disjoint $K_1, G_1 \subseteq [n]$ 
and a diagram pair $d$,
we can uniquely recover $D$.
Thus, we may identify
a diagram pair with its image and write
$D = (K_1, G_1, d)$.

Fix a weak composition $\alpha$ in this section.
An element of $\KKD(\alpha)$
or $\RSVT(\alpha)$
can be written as $(K_1, G_1, d)$.
We will find conditions on this triple 
to determine when 
$(K_1, G_1, d) \in \KKD(\alpha)$ and $\RSVT(\alpha)$.

\subsection{Describing $\KKD(\alpha)$}

First, we describe the condition for a diagram pair 
$(K_1, G_1, d)$ to live in $\KKD(\alpha)$.
\begin{thm}

\label{T: KKD image}
The diagram pair $(K_1, G_1, d)$ is in $\KKD(\alpha)$ 
if and only if it satisfies:
\begin{enumerate}
\item $K_1$ and $G_1$ are 
disjoint subsets of $[n]$.
\item 
$K_1 \leqslant \supp(\alpha) $.
\item 
For each $g \in G_1$,
$|[g,n] \cap \supp(\alpha)| > |[g,n] \cap K_1|$.
\item 
$d \in \KKD(\overline{M(\alpha, K_1)})$
\end{enumerate}
\end{thm}

The rest of this subsection proves it. 
First, 
we want to show if a diagram pair $D$ satisfies 
these conditions, then $D \in \KKD(\alpha)$.
This can be implied from the following lemma:
\begin{lemma}
\label{L: Reverse K-Kohnert}
If $D = (K_1, G_1, d)$ satisfies 
conditions (1)-(4) in Theorem~\ref{T: KKD image},
then we can find $\gamma \leqslant \alpha$
such that $D \in \KKD(\gamma)$.
Moreover, $\gamma_i = 0$ 
if $i \notin K_1 \sqcup G_1$.
\end{lemma}

Notice $D \in \KKD(\gamma) \subseteq \KKD(\alpha)$
by Corollary~\ref{C: KD RSVT order}.
Thus, this lemma implies the reverse direction
of Theorem~\ref{T: KKD image}.

\begin{proof}
We describe an algorithm that turns $D$
into a key diagram pair of some weak composition $\gamma$
via reverse K-Kohnert moves.
First, we consider $d$.
Since $d \in \KKD(\overline{M(\alpha, K_1)})$,
we may do reverse K-Kohnert moves on $d$
to obtain the key diagram pair of 
$\overline{M(\alpha, K_1)}$.

Now, Kohnert cells in $D$ form the key diagram pair
of $M(\alpha, K_1)$.
If $G_1$ is empty, we are done.
Otherwise, let $g = \min(G_1)$.
By Lemma~\ref{L: Lift M},
we can find $k \in K_1 \cap [1, g)$ such that 
$M(\alpha, (K_1 \sqcup \{g\}) - \{k\} )
= (g \:\: k) M(\alpha, K_1)$.
We perform reverse K-Kohnert moves
to lift the entire row $k$ of $D$ to row $g$, and remove the ghost at $(1,g)$. Next we assign $(K_1 \sqcup \{g\}) - \{k\}$ to $K_1$
and assign $G_1 - \{g\}$ to $G_1$.
Now, Kohnert cells in $D$ still form 
the key diagram pair of $M(\alpha, K_1)$.
We may repeat the steps above until $G_1$ is empty.
The resulting diagram pair is a key diagram pair
for some weak composition $\gamma$.
Clearly, $\gamma_i = 0$ 
if $(1, i)$ was not a cell in $D$. 
\end{proof}

\begin{proof}[Proof of Theorem~\ref{T: KKD image}]

The reverse direction is already shown. 
For each $D = (K_1, G_1,d) \in KKD(\alpha)$, 
we need to show
the four conditions are satisfied.
We prove by induction on $\max(\alpha)$.
When $\max(\alpha) = 0$,
the claim is immediate.

Assume the statement is true
for any weak composition whose maximum entry
is less than $\max(\alpha)$.
For any $D \in \KKD(\alpha)$,
we want to check it satisfies 
the four conditions.
Condition (1) is immediate. 
We prove the other conditions
by an induction on K-Kohnert moves.
If $D$ is the key diagram pair of $\alpha$,
then the last three conditions are immediate.
Now assume $(K_1, G_1, d) \in \KKD(\alpha)$
satisfies the last three conditions. 
Perform one K-Kohnert move and obtain $(K_1', G_1', d')$.
We want to show $(K_1', G_1', d')$ also satisfies
the last three conditions. 
If the K-Kohnert move is not on column 1,
then $K_1 = K_1'$ and $G_1 = G_1'$,
which gives us the second and the third condition.
Notice that $d'$ is obtained from $d$ by one K-Kohnert move,
so the last condition is also satisfied. 

Now suppose the K-Kohnert move is on column 1. 
$K_1'= (K_1 - \{i\}) \sqcup \{j\}$ with $i > j$, and $G_1'$ is either $G_1$ or $G_1 \sqcup \{ i\}$. We check the last three conditions below.
\begin{enumerate}
    \setItemnumber{2}
    \item $\supp(\alpha) \geqslant K_1 \geqslant K_1'$.
    \item We check $|[g,n] \cap \supp(\alpha)| > |[g,n] \cap K_1'|$,
for each $g \in G_1 \sqcup \{ i\}$. \\If $g \in G_1$, we have $$|[g,n] \cap \supp(\alpha)| > |[g,n] \cap K_1| \geqslant |[g,n] \cap K_1'|.$$
For $g = i$, we have  $$|[i,n] \cap \supp(\alpha)| \geqslant |[i,n] \cap K_1| > |[i,n] \cap K_1'|.$$
    \item We have $d' = d \in \KKD(\overline{M(\alpha, K_1)})$.
By our inductive hypothesis,
$d$ satisfies all four conditions of $\KKD(\overline{M(\alpha, K_1)})$.
Using Lemma~\ref{L: Reverse K-Kohnert},
we may perform reverse K-Kohnert moves on $d$
and get the key diagram pair of 
$\gamma \leqslant \overline{M(\alpha, K_1)}$.
Since we can make a K-Kohnert move on $(1,i)$,
$d$ has no cells on row $i$.
Thus, we know $\gamma_i = 0$.
To show $d \in \KKD(\overline{M(\alpha, K_1')})$, 
we show $\gamma \leqslant  \overline{M(\alpha, K_1')}$. 

Let $T = \key(M(\alpha,K_1))$.
For each column of $T$ that contains $i$,
we replace $i$ by the largest $i'$ such that 
$i' < i$ and $i'$ is not in this column.
Then we sort the column into strictly decreasing order.
Let $T'$ be the resulting tableau. 
It is clear that $T'$ is a key.
Let $\sigma = \wt(T')$. 
We make two observations about $T'$ and $\sigma$:
\begin{itemize}
\item Column 1 of $T'$ consists of $K_1'$,
so $\supp(\sigma) = K_1'$.
\item The tableau $T'$ is entry-wise less than $T$.
Thus, $\sigma \leqslant M(\alpha,K_1) \leqslant \alpha$.
\end{itemize}
By definition, these two observations yield
$\sigma \leqslant M(\alpha, K_1')$.

It remains to check $\gamma \leqslant \overline{\sigma}$.
Let $g_1> \dots > g_s$ be the numbers 
in column $c$ of $\key(\gamma)$.
By $\gamma_i = 0$, none of these numbers is $i$.
Let $t_1 > \dots > t_s$ (resp. $t_1' > \dots > t_s'$)
be numbers in column $c+1$ of $T$ (resp. $T'$).
By $\gamma \leqslant \overline{M(\alpha, K_1)}$,
we know $g_k \leqslant t_k$ for $1 \leqslant k \leqslant s$.
If $i$ is not in column $c+1$ of $T$,
then we are done. 
Otherwise, assume $t_a = i$.
We know $t_1',\dots, t_s'$
are obtained from $t_1, \dots, t_s$ by changing
$t_a, \cdots, t_b$ into $t_a - 1, \dots, t_b - 1$.
Since $g_a \leqslant i$ and $g_a \neq i$,
we have $g_a \leqslant i - 1 = t_{a}'$.
Then $g_{a+1} \leqslant i - 2 = t_{a + 1}'$.
Following this argument, 
we have $g_k \leqslant t_k'$ for $1 \leqslant k \leqslant s$.
Thus, $\gamma \leqslant \overline{\sigma} 
\leqslant M(\alpha, K_1')$.
\end{enumerate}
\end{proof}

\subsection{Describing $\RSVT(\alpha)$}
Take $T \in \RSVT(\alpha)$
and write $T$ as $(L_1, E_1, t)$.
First, we notice the following relation
between $K_-(T), K_-(t)$ and $L_1$:
\begin{lemma}
\label{L: Left key and m}
Assume $T = (L_1, E_1, t) \in \RSVT(\alpha)$. 
Assume $K_-(t) = \key(\gamma)$.
Then $\wt(K_-(T)) = \mathbbm{1}_{L_1}
+ m(\gamma, L_1)$.
\end{lemma}
\begin{proof}
Let $S = \key(\mathbbm{1}_{L_1}
+ m(\gamma, L_1))$. 
View $S$ and $T$ as tableaux.
Let $T_c$ be the set consisting of 
leading numbers in column $c$ of $T$,
so $T_1 = L_1$.
It suffices to show that $K_-(T) = S$.
We compare these two tableaux column by column. 
Apparently, column 1 of $K_-(T)$ and column 1 of $S$ 
both consist of $L_1$.

Consider the column $c$ of $K_-(T)$ and $S$ for $c>1$. 
Column $c$ of $S$
agrees with column $c-1$ of $\key(m(\gamma, L_1))$.
Let $\key(\gamma)_{c-1}$ consists
of numbers in column $c-1$
of $\key(\gamma)$.
By Lemma~\ref{lem:algorithm_lhd},
column $c-1$ of $\key(m(\gamma, L_1))$ 
consists of $L_1 \lhd \key(\gamma)_{c-1}$.
Since $\key(\gamma) = K_-(t)$, 
we have $\key(\gamma)_{c-1} = T_2 \lhd \dots \lhd T_c$.
Thus, column $c$ of $S$ consists of 
$L_1 \lhd T_2 \lhd \dots \lhd T_c = T_1 \lhd T_2 \lhd \dots \lhd T_c$,
which agrees with column $c$ of $K_-(T)$.
\end{proof}

\begin{thm}
\label{T: RSVT image}
The diagram pair $(L_1,E_1,t)$ is in $\RSVT(\alpha)$ 
if and only if it satisfies:
\begin{enumerate}
\item $L_1$ and $E_1$ are 
disjoint subsets of $[n]$.
\item 
$L_1 \leqslant \supp(\alpha)$.
\item 
Let $L_1'$ be the set of 
row indices of Kohnert cells in column 1 of $t$.
For each $e \in E_1$,
$|(e, n] \cap L_1| > |(e, n] \cap L_1'|$.
\item 
$t \in \RSVT(\overline{M(\alpha, L_1)})$
\end{enumerate}
\end{thm}

\begin{proof}
First, we show that if 
$T = (L_1, E_1, t) \in \RSVT(\alpha)$,
then the four conditions are satisfied.
\begin{enumerate}
\item $L_1$ and $E_1$ are clearly disjoint.
\item Notice that column 1 of $K_-(T)$ 
consists of numbers in $L_1$,
while column 1 of $\key(\alpha)$ 
consists of numbers in $\supp(\alpha)$.
By $K_-(T) \leqslant \key(\alpha)$,
$L_1 \leqslant \supp(\alpha)$.

\item View $T$ as a tableau.
For $e \in E_1$, assume it is on row $j$ column $1$
of $T$.
Then $|L_1 \cap (e, n]| = j$.
The set $L_1'$ consists of
leading numbers in column 2 of $T$.
The $j^{th}$ largest number in $L_1'$,
if exists,
is the leading number 
at row $j$ column 2 of $T$.
Thus, it is weakly less than $e$,
so $|(e, n] \cap L_1'| < j$.

\item
By $T \in \RSVT(\alpha)$,
$K_-(T) \leqslant \key(\alpha)$.
Let $\gamma = \wt(K_-(t))$.
By Lemma~\ref{L: Left key and m},
$\mathbbm{1}_{L_1} + m(\gamma, L_1) \leqslant \alpha$.
By Lemma~\ref{L: m and M},
$$
\wt(K_-(t)) = \gamma
\leqslant \overline{M(\alpha, L_1)}.$$
Thus $t \in \RSVT(\overline{M(\alpha, L_1)})$. 
\end{enumerate} 

Now, we check if we have $T = (L_1,E_1, t)$ 
satisfying these four conditions,
then $T \in \RSVT(\alpha)$.
We may construct $T$ as a tableau.
First, we build column 1 of $T$.
We arrange numbers in $L_1$ into a 
strictly decreasing column.
For each $e \in E_1$,
by condition 2, $|(e, n] \cap L_1| > 0$.
Then we put $e$ on row $|(e, n] \cap L_1|$.
Clearly, this column is an RSVT with 
leading numbers from
$L_1$ and extra numbers from $E_1$.

Next, we may build the tableau 
corresponding to $t$ recursively.
Let $L_1'$ be the set of leading numbers 
in column 1 of $t$.
Put the column we just constructed 
on the left of $t$.
Only need to check all numbers in row $j$
of our new column are weakly larger than
the leading number in row $j$ column 1 of $t$,
which is the $j^{th}$ largest number in $L_1'$.
\begin{itemize}
\item Let $e$ be an extra number 
on row $j$ of our new column, 
By condition 3,
$|[e, n] \cap L_1'| < j$.
Thus, the $j^{th}$ largest number in $L_1'$
is at most $e$.
\item By condition 4, if we let $\gamma = \wt(K_-(t))$,
then $\gamma \leqslant \overline{M(\alpha, L_1)}$.
By Lemma~\ref{L: m and M}, 
$m(\gamma, L_1)$ exists. 
Then by Corollary~\ref{C: m exists},
the $j^{th}$ largest number 
in $\supp(\gamma) = L_1'$
is weakly less than the $j^{th}$ largest number
in $L_1$, which is the leading number in row
$j$ of our new column.
\end{itemize}

Now we have constructed a tableau $T$ which if viewed 
as a diagram pair,
corresponds to $(L_1, E_1, t)$.
It remains to check $K_-(T) \leqslant key(\alpha)$.
Notice that $K_-(T) = \key(\mathbbm{1}_{L_1}
+ m(\gamma, L_1))$.
By Lemma~\ref{L: m and M} and $\gamma \leqslant \overline{M(\alpha, L_1)}$,
we have $K_-(T) \leqslant key(\alpha)$.
Thus, $(L_1, E_1, t) \in \RSVT(\alpha)$.
\end{proof}

The recursive descriptions of $\KKD(\alpha)$
and $\RSVT(\alpha)$ share many similarities.
They only differ at the third condition,
which is the condition on the positions
of ghost cells. 
From this observation,
we get the following result which was discussed in Remark~\ref{natual_map}.

\begin{cor}
We have $\RSSYT(\alpha) = \KD(\alpha)$.
\end{cor}
\begin{proof}
We know $\KD(\alpha)$ is the subset of $\KKD(\alpha)$ 
containing all diagram pairs with no ghost cells. 
By Theorem~\ref{T: KKD image},
$\KD(\alpha)$ consists of $(K_1, \emptyset, d)$
such that 
\begin{itemize}
\item $K_1 \leqslant \supp(\alpha)$.
\item $d \in \KD(\overline{M(\alpha, K_1)})$.
\end{itemize}

On the other hand, 
$\RSSYT(\alpha)$ is the subset of $\RSVT(\alpha)$ 
containing all diagram pairs with no ghost cells. 
By Theorem~\ref{T: RSVT image},
$\RSSYT(\alpha)$ consists of $(L_1, \emptyset, t)$
such that 
\begin{itemize}
\item $L_1 \leqslant \supp(\alpha)$.
\item $t \in \RSSYT(\overline{M(\alpha, L_1)})$.
\end{itemize}

An induction on $\max(\alpha)$ yields
$\KKD(\alpha) = \RSSYT(\alpha)$.
\end{proof}

This recursive description 
of $\RSVT(\alpha)$ leads to
the following lemma. 

\begin{lemma}
\label{L: RSVT in another set}
Let $T = (L, E) = (L_1, E_1, t)$
be an element in $\RSVT(\alpha)$.
If $(L, \emptyset) \in KD(\gamma)$
for another weak composition $\gamma$,
then $T \in \RSVT(\gamma)$.
\end{lemma}
\begin{proof}
Prove by induction on $\max(\gamma)$.
Notice that $t$ is in
$\RSVT(\overline{M(\alpha, K_1)})$.
If we ignore ghost cells of $t$,
it is in $\KD(\overline{M(\gamma, K_1)})$.
By our inductive hypothesis, 
$t \in \RSVT(\overline{M(\gamma, K_1)})$.

Now we check $T = (L_1, E_1, t)$ 
satisfies the four
conditions of $\RSVT(\gamma)$.
Condition 1 and 3 are implied by 
$T \in \RSVT(\alpha)$.
Condition 2 follows from $(L, \emptyset)
\in \KD(\gamma)$.
The last condition is checked in the 
previous paragraph.
\end{proof}

\begin{rem}
Alternatively, we may prove 
Lemma~\ref{L: RSVT in another set}
while viewing $T$ as a tableau.
Let $T'$ be the tableau
we get after keeping the smallest
number in each cell of $T$.
By $(L, \emptyset) \in \KD(\gamma)$,
we know $T' \in \RSSYT(\gamma)$.
Consequently, 
$K_-(T) = K_-(T') \leqslant \gamma$,
so $T \in \RSVT(\gamma)$.
\end{rem}

\section{Two operators on Kohnert diagrams}
\label{section.flat.sharp}
In order to prove the well-definedness and bijectivity of $\Psi_\alpha$ and $\Phi_\alpha$ 
defined in Section~\ref{section.maps}, we introduce two auxiliary operators $\sharp_g$ and $\flat_e$ on $\KD(\alpha)$ and study their properties. Later in Section~\ref{section.recursive}, we will use these two operators to give alternative
descriptions of $\Psi_\alpha$ and $\Phi_\alpha$.

\subsection{Introducing the $\sharp_g$ operator}
We define an operator $\sharp_g$ on $\KD(\alpha)$
for each $g \in [n]$.
\begin{defn}
For each $g \in [n]$,
define $\sharp_g: \KD(\alpha) \rightarrow 
\KD(\alpha) \times [n]$.
Take $D = (K_1, \emptyset, d) \in \KD(\alpha)$.
Find the largest $k \in K_1 \cap [1, g]$ such that
$D' = (K_1 - \{k\} \sqcup \{g\}, \emptyset, d)$ is still 
in $\KD(\alpha)$.
If such $k$ exists, 
then $\sharp_g(D) := (D', k)$.
Otherwise, $\sharp_g(D)$ is undefined.
\end{defn}

We would like to determine when $\sharp_g(D)$
is well-defined.
This is partially answered by the following lemma:
\begin{lemma}
\label{L: sharp well-defined0}
Assume $D = (K_1, \emptyset, d) \in \KD(\alpha)$.
If $\sharp_g(D)$ is defined,
then $|[g, n] \cap \supp(\alpha)| 
> |(g, n] \cap K_1|$.
\end{lemma}
The proof involves Kohnert tableaux
from subsection~\ref{Kohnert Tableaux}.
\begin{proof}
Assume $\sharp_g(D) = (D', k)$.
Then in column 1 of $D'$,
there are $|(g, n] \cap K_1| + 1$
cells weakly above row $g$.
In $\Label_\alpha(D')$,
these cells are filled by distinct 
number in $[g, n] \cap \supp(\alpha)$.
Thus, 
$|[g, n] \cap \supp(\alpha)| 
> |(g, n] \cap K_1|$.
\end{proof}

Next,
we will show the converse of this lemma.
First, we introduce an algorithm called 
\definition{sharp algorithm}.
Its input would be a number $g$ and 
$D = (K_1, \emptyset, d) \in \KD(\alpha)$
such that $|[g, n] \cap \supp(\alpha)| 
> |(g, n] \cap K_1|$.
It will output a diagram pair $D'$
with only Kohnert cells. 
It will also output a filling of $D'$.
Later, we will check the filling is a
Kohnert tableau with content $\alpha$,
which implies $D' \in \KD(\alpha)$.
Finally, we will check $D'$ is 
the first component of $\sharp_g(D)$.

The sharp algorithm consists of five steps:
\begin{itemize}
\item Step 1: Compute $\Label_\alpha(D)$.
\item Step 2: Since $|[g, n] \cap \supp(\alpha)| > |(g, n] \cap K_1|$, there is a number $m$ such that 
$m \geqslant g$ but $m$ is weakly below row $g$ in column 1. Find the highest such $m$. 
Let $k$ be the row index of this $m$.
\item Step 3: Let $D' = ((K_1 - \{k\})\sqcup \{g\}, \emptyset, d)$.
This is the first output.
\item Step 4: To compute the filling, 
we start from $\Label_\alpha(D)$ 
and move the $m$ from $(1, k)$ to $(1,g)$.
The resulting filling satisfies the
first three conditions from Definition~\ref{def: KT}. 
\item Step 5: If there is an $u < m$ such that $u, m$ 
violates condition four in column 1,
we find the smallest such $u$ and swap it
with $m$.
Repeat this step until 
no such $u$ exists.
The final filling will be the second output. 
\end{itemize}

\begin{exa} \label{eg_sharp} Consider $\alpha = (0,0,0,2,2,1,1)$, and $D\in \KD(\alpha)$ as shown below. Let $g = 3$. The sharp algorithm gives $m = 6$ and $k = 1$. 
The output $D'$ is obtained by moving the Kohnert cell $(1, k)$
to $(1, g)$ in $D$.
To obtain the filling, we need to first move $m$ to row $g$ in $\Label_\alpha(D)$. 
Next, swap $m$ with $4$ and then $5$. 
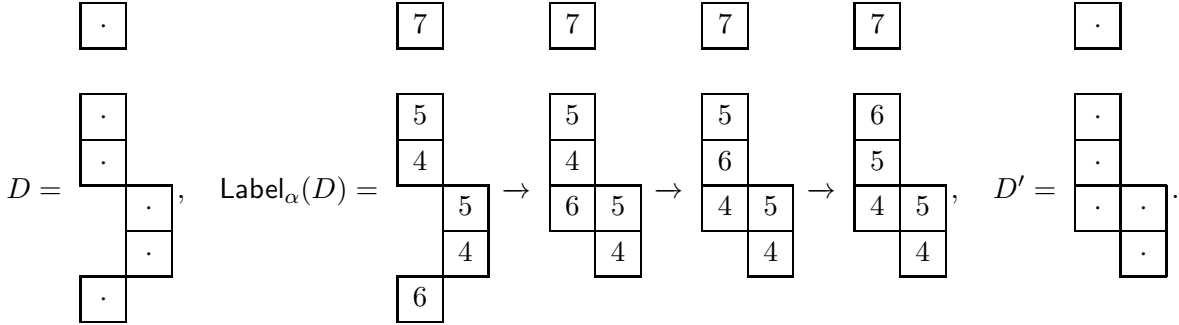
\[
D = \raisebox{2cm}{
\begin{ytableau}
\cdot \\
\none \\
\cdot\\
\cdot \\
\none & \cdot \\
\none & \cdot \\
\cdot 
\end{ytableau}}\,,  \quad
\Label_\alpha(D) = \raisebox{2cm}{
\begin{ytableau}
7 \\
\none \\
5\\
4 \\
\none & 5 \\
\none & 4 \\
6 
\end{ytableau}}\,\xrightarrow[]{}
\raisebox{2cm}{
\begin{ytableau}
7 \\
\none \\
5\\
4 \\
6 & 5 \\
\none & 4 \\
\none 
\end{ytableau}}\,\xrightarrow[]{}
\raisebox{2cm}{
\begin{ytableau}
7 \\
\none \\
5\\
6 \\
4 & 5 \\
\none & 4 \\
\none 
\end{ytableau}}\,\xrightarrow[]{}
\raisebox{2cm}{
\begin{ytableau}
7\\
\none \\
6\\
5 \\
4 & 5 \\
\none & 4 \\
\none
\end{ytableau}}\,, \quad D' = 
\raisebox{2cm}{
\begin{ytableau}
\cdot \\
\none \\
\cdot\\
\cdot \\
\cdot & \cdot \\
\none & \cdot \\
\none 
\end{ytableau}}\,.
\]
\end{exa}
\begin{lemma}
The filling produced by the sharp algorithm
is a Kohnert tableau with content $\alpha$.
Consequently, 
the filling is $\Label_\alpha(D')$
and $D' \in \KD(\alpha)$.
\end{lemma}
\begin{proof}
We claim that after Step 4 and after each  
iteration of Step 5, 
the filling satisfies the
first three conditions of Definition~\ref{def: KT}. 
Moreover, 
if $i< j$ violates the last condition
in column $c$, then $j = m$ and $c = 1$.

After Step 4,
the filling clearly satisfies
the first three conditions. 
Now assume $i<j$ violates 
the last condition in column $c$.
Clearly, $c = 1$ 
and $m$ is $i$ or $j$.
Assume $m$ is $i$, then $j$ is below row $g$ in column 1.
If $j$ were below row $k$,
then $m, j$ would have violated condition 4 
before this move.
On the other hand, if $j$ were above row $k$,
then we would have picked $j$ instead of $m$.
In either case, we reach a contradiction.
Thus, $j$ must be $m$.

If there is a $u$ such that 
$u< m$ violates condition 4,
we pick the smallest such $u$.
Assume our claim holds now.
We need to show our claim is still
true after we swap $u$ and $m$. 
We check the first three conditions:
\begin{enumerate}
\item Condition 1 clearly holds.
\item Only need to check condition 2 for $m$.
Recall $u < m$.
Since $u$ satisfies condition 2 
before the move, 
so does $m$ after the move.
\item Only need to check condition 3 for $u$.
Since $u, m$ violate condition 4 before this swap, and that $u$ satisfies condition 3 before moving $m$ to $(1,g)$,
there is no $u$ in column 2
strictly above the $m$ in column 1. 
Thus, after the move, 
$u$ satisfies condition 3.

\end{enumerate}
Now assume $i< j$ violates 
condition 4 in column $c$.
Clearly, $c = 1$
and one of $i, j$ is $u$ or $m$.
We just need to check $u$ cannot be $i$ 
or $j$ and $m$ cannot be $i$:
\begin{enumerate}
\item[$\bullet$] Assume $i = u$.
Then $u, j$ would have violated condition 4
before the move, 
contradicting to our claim.
\item[$\bullet$] Assume $j = u$.
Then $i < u < m$.
Before the move, 
$i, m$ violates condition 4.
Then we would have picked $i$ and swapped it with $m$, rather than $u$.
Contradiction. 

\item[$\bullet$] Assume $i = m$.
If $j$ were below $m$ before the move,
then $m, j$ would have violated condition 4 before the move.
Now assume $j$ were between $u$ and $m$ before the move.
Then $u, j$ would have violated condition 4
before the move.
\end{enumerate}

Now our claim holds
after each move. 
When the sharp algorithm terminates,
there is no violation of condition 4
with the form $i, m$.
Thus, the filling satisfies condition 4,
so it is in $\KT(\alpha)$.
\end{proof}
 
\begin{lemma}
The $D'$ yielded by the sharp algorithm 
is the first component of $\sharp_g(D)$.
\end{lemma}

\begin{proof}
The lemma is trivial if $g \in K_1$.
Thus, we may assume $g \notin K_1$.
We already showed $D' \in \KD(\alpha)$
by constructing $\Label_\alpha(D')$.
Take $r \in K_1$ with $k < r < g$.
It suffices to show $(K_1 - \{r\} \sqcup \{g\},
\emptyset, d) \notin \KD(\alpha)$.

In column 1 of $\Label_\alpha(D)$,
assume $s_1, \dots, s_p$ are
the numbers below row $g$ and 
weakly above row $r$.
By how we picked $m$,
we know $s_1, \dots, s_p < g$.
We may run the labeling algorithm 
on $((K_1 - \{r\}) \sqcup \{g\},
\emptyset, d)$.
It behaves the same as on $D$ 
on cells prior to $(1, r)$.
After filling all these cells before $(1,r)$, 
we know $s_1, \dots, s_p$ are still 
in the set $S_1$.
However, there remains only $p - 1$ 
empty cells below row $g$.
Thus, at least one number of $s_1, \dots, s_p$
will be placed weakly above row $g$.
Since this number is less than $g$,
the labeling algorithm will terminate and produce no output.
Thus, this diagram is not in $\KD(\alpha)$.
\end{proof}

Thus, we know the sharp algorithm 
outputs the first component of $\sharp_g(D)$,
together with its Kohnert Labeling. 
Now we can tell when $\sharp_g(D)$ is well-defined. 
\begin{lemma}
\label{L: sharp well-defined}
Assume $D = (K_1, \emptyset, d) \in \KD(\alpha)$.
Then $\sharp_g(D)$ is well-defined if and only if 
$$|[g, n] \cap \supp(\alpha)| 
> |(g, n] \cap K_1|.$$
\end{lemma}
\begin{proof}
The forward direction is given 
by Lemma~\ref{L: sharp well-defined0}.
The other direction follows from the sharp algorithm.
\end{proof}

\begin{cor}
\label{C: numbers fixed by sharp}
Take $D \in \KD(\alpha)$.
Assume $\sharp_g(D) = (D', k)$.
Assume $(1,k)$ is filled by $m$ 
in $\Label_\alpha(D)$. 
Then $\Label_\alpha(D)$ and $\Label_\alpha(D')$
agree at the cell $(c,r)$, if $(c,r)$ satisfies
one of the following:
\begin{enumerate}
\item[$\bullet$] $c > 1$;
\item[$\bullet$] $r < g$ and $r \neq k$;
\item[$\bullet$] $r > m$.
\end{enumerate}
\end{cor}
\begin{proof}
By the behavior of the sharp algorithm,
$\Label_\alpha(D')$ is obtained from
$\Label_\alpha(D)$ by moving $m$ from $(1,k)$ 
to $(1,g)$ 
and repeatedly swapping $m$ with a number above it.
The number $m$ will not go above row $m$,
so only $(1,k)$ and cells between row $g$ and 
row $m$ in column 1 are affected.
\end{proof}

\subsection{Commutativity of $\sharp_g$ operators}
Next, we observe that two $\sharp_g$
operators might ``commute'' under certain conditions.
Consider the following example:
\begin{exa}
Let $\alpha = (0,0,2,1)$.
Let $D$ be the following element in $\KD(\alpha)$.
We first apply $\sharp_3$ and get $(D^1, 2)$.
Then apply $\sharp_4$ on $D^1$ and get $(D^{final}, 1)$.
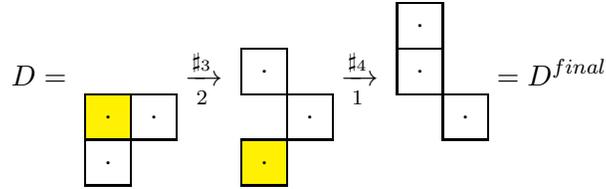
\[
D = 
\raisebox{0.5cm}{
\begin{ytableau}
\none \\
\none \\
*(yellow)\cdot & \cdot \\
\cdot \\
\end{ytableau}}
\xrightarrow[2]{\sharp_{3}}
\raisebox{0.5cm}{
\begin{ytableau}
\none \\
\cdot \\
\none & \cdot \\
*(yellow)\cdot \\
\end{ytableau}}
\xrightarrow[1]{\sharp_{4}}
\raisebox{0.5cm}{
\begin{ytableau}
\cdot \\
\cdot \\
\none & \cdot \\
\none \\
\end{ytableau}}
= D^{final}
\]

We can try to swap the order 
of these two operators.
We first apply $\sharp_4$ on $D$
and get $(D^2, 1)$.
Then we apply $\sharp_3$ on $D^2$
and get $(D^{final}, 2)$.
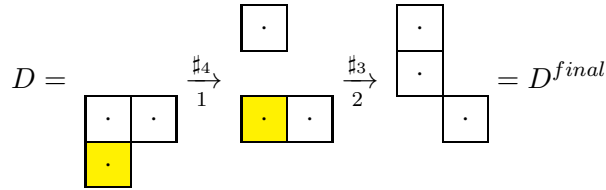
\[
D = 
\raisebox{0.5cm}{
\begin{ytableau}
\none \\
\none \\
\cdot & \cdot \\
*(yellow)\cdot \\
\end{ytableau}}
\xrightarrow[1]{\sharp_{4}}
\raisebox{0.5cm}{
\begin{ytableau}
\cdot \\
\none \\
*(yellow)\cdot & \cdot \\
\none \\
\end{ytableau}}
\xrightarrow[2]{\sharp_{3}}
\raisebox{0.5cm}{
\begin{ytableau}
\cdot \\
\cdot \\
\none & \cdot \\
\none \\
\end{ytableau}}
= D^{final}
\]

Observe that changing the order of these
two operators will not affect the final Kohnert diagram.
\end{exa}

This phenomenon is captured by the following two lemmas. 
\begin{lemma}
\label{L: sharp commute 1}
Take $D \in \KD(\alpha)$.
Take $g_1, g_2 \in [n]$ with $g_1 < g_2$.
Assume $\sharp_{g_1}(D) = (D^1, k_1)$
and
$\sharp_{g_2}(D^1) = (D^{final}, k_2)$.
If $k_1 > k_2$,
then the two operators ``commute''.
That is:
\begin{enumerate}
\item[$\bullet$] $\sharp_{g_2}(D) = (D^2, k_2)$ 
for some $D^2 \in \KD(\alpha)$, and
\item[$\bullet$] $\sharp_{g_1}(D^2) = (D^{final}, k_1)$.
\end{enumerate}
\end{lemma}
\begin{proof}
Let $C$ be the first column 
of $\Label_\alpha(D)$.
Define $C^1$ and $C^{final}$ similarly. 
In $C$, 
let $m_1$ (resp. $m_2$) be the
number at row $k_1$ (resp. $k_2$).

First, we claim $m_1$ is below 
row $g_2$ in $C^1$.
If not, 
then we may find a number $u$ 
that is weakly above row $g_2$ in $C$
but below row $g_2$ in $C^1$.
By Corollary~\ref{C: numbers fixed by sharp},
the $u$ is above row $g_1$ in $C^1$,
so $u$ is higher than $m_2$.
By $u \geqslant g_2$,
we should pick $u$ rather than $m_2$
when computing $\sharp_{g_2}(D^1)$.
Contradiction. 
 
Now consider $C^1$.
By Corollary~\ref{C: numbers fixed by sharp},
$m_2$ is still at row $k_2$.
All numbers between $m_2$ and row $g_2$ 
will be less than $g_2$.
Thus, $m_1 < g_2$.
Since $C$ and $C^1$ only differ 
between row $g_1$ and row $m_1$,
all numbers between $m_2$ and row $g_2$ 
will be less than $g_2$ in $C$.
Thus, $\sharp_{g_2}$ will also pick $m_2$ 
when acting on $D$.
$\sharp_{g_2}(D) = (D^2, k_2)$.
Column 1 of $\Label_\alpha(D^2)$
agrees with $C$ between row $k_1$
and $g_1$.
Thus, $\sharp_{g_1}$ will pick $m_1$
when acting on $D^2$.
\end{proof}

\begin{lemma}
\label{L: sharp commute 2}
Take $D \in \KD(\alpha)$.
Take $g_1, g_2 \in [n]$ with $g_1 < g_2$.
Assume $\sharp_{g_2}(D) = (D^2, k_2)$
and
$\sharp_{g_1}(D_2) = (D^{final}, k_1)$.
If $k_1 > k_2$,
then the two operators ``commute''.
That is:
\begin{enumerate}
\item[$\bullet$] $\sharp_{g_1}(D) = (D^1, k_1)$ 
for some $D^1 \in \KD(\alpha)$, and
\item[$\bullet$] $\sharp_{g_2}(D^1) = (D^{final}, k_2)$.
\end{enumerate}
\end{lemma}
\begin{proof}
Let $C$ be the first column 
of $\Label_\alpha(D)$.
Define $C^2$ and $C^{final}$ similarly. 
In $C$, 
let $m_1$ (resp. $m_2$) be the
number at row $k_1$ (resp. $k_2$).

First, $C$ and $C^2$ agree
between row $k_2$ and row $g_2$.
Thus, $\sharp_{g_1}$ would also 
pick $m_1$ when acting on $D$,
so $\sharp_{g_1}(D) = (D^1, g_2)$.

Since $\sharp_{g_2}$ picks
$m_2$ when acting on $D$,
all numbers between row $k_2$ and row $g_2$
in $C$ are less than $g_2$.
In particular, $m_1 < g_2$.
We know column 1 of $D^1$ is obtained
from $C$ by changing cells 
between row $k_1$ and row $m_1$.
Thus, in column 1 of $D^1$,
all numbers between row $k_2$ and row $g_2$ are still less than $g_2$.
When acting on $D^1$, 
$\sharp_{g_2}$ would still pick $m_2$.
\end{proof}

\subsection{Introducing the $\flat_k$ operator}
Next, we define the operator $\flat_k$,
which can be viewed as the (partial) inverse of $\sharp_g$.

\begin{defn}
For each $k \in [n]$,
define $\flat_k: \KD(\alpha) \rightarrow 
\KD(\alpha) \times [n]$.
Take $D = (K_1, \emptyset, d) \in \KD(\alpha)$.
Find the smallest $g \in K_1 \cap [k, n]$ such that
$D' = ((K_1 - \{g\}) \sqcup \{k\}, \emptyset, d)$ is still 
in $\KD(\alpha)$.
If such $g$ exists, 
then $\flat_k(D) := (D', k)$.
Otherwise, $\flat_k(D)$ is undefined.
\end{defn}

\begin{lemma}
\label{L: sharp and flat are inverses}
Take $D \in \KD(\alpha)$.
\begin{enumerate}
\item[$\bullet$] Take $g \in [n]$. 
If $\sharp_g(D) = (D', k)$,
then $\flat_k(D') = (D, g)$.
\item[$\bullet$] Take $k \in [n]$. 
If $\flat_k(D) = (D', g)$,
then $\sharp_g(D') = (D, k)$.
\end{enumerate}
In other words,
$\flat_k$ and $\sharp_g$ are (partial) inverses of each other.
\end{lemma}
\begin{proof}
Assume $D = (K_1, \emptyset, d)$.
Consider the first statement.
If $g \in K_1$, 
then $k = g$ and $\flat_k(D') = (D, g)$ trivially.
Now assume $g \notin K_1$. 
Then $((K_1 - \{r\}) \sqcup \{g\} , \emptyset, d)$
is not in $\KD(\alpha)$ for all $r \in K_1$ with $k < r < g$.
Thus, $\flat_k(D') = (D, g)$.
The second statement can be proved similarly.
\end{proof}

We would like to determine when $\flat_k(D)$
is well-defined.
This is answered by the following lemma:
\begin{lemma}
\label{L: flat well-defined}
Assume $D = (K_1, \emptyset, d) \in \KD(\alpha)$.
Let $K_2$ be the set of row indices for cells
in column 2 of $D$.
Then $\flat_k(D)$ is well-defined if and only if $k \in K_1$
or
$|K_1 \cap (k, n]| 
> |K_2 \cap (k, n]|$.
\end{lemma}
\begin{proof}
First, assume the condition fails.
We show $\flat_k(D)$ is undefined.
Assume by contradiction that 
$\flat_k(D) = (D', g)$.
Define $K_1'$ and $K_2'$ similarly for $D'$.
Then $K_1' = (K_1 - \{g\}) \sqcup \{k\}$
and $K_2' = K_2$.
Thus, 
$$
|K_1' \cap (k, n]| < |K_1 \cap (k, n]|
\leqslant |K_2 \cap (k, n]| = |K_2' \cap (k, n]|.
$$
Then consider $\Label_\alpha(D')$.
There are $|K_2' \cap (k, n]|$ distinct
numbers above row $k$ in column 2.
They all must appear above row $k$ 
in column 1,
but there are not enough cells for them.
Contradiction.

Now assume the condition holds,
we show $\flat_k(D)$ is well-defined.
Clearly, we are done if $k \in K_1$.
Now assume $|K_1 \cap (k, n]| 
> |K_2 \cap (k, n]|$ and
consider $\Label_\alpha(D)$.
By our assumption,
we can find $m$ above row $k$ in column 1
such that there is no $m$ 
above row $k$ in column 2.
Pick the lowest such $m$ and move it to $(1, k)$.
Then the resulting filling is in $\KT(\alpha)$:
\begin{enumerate}
\item Condition 1 of $\KT(\alpha)$ is clear.
\item Since we moved a cell down, 
condition 2 is clear.
\item Condition 3 holds for $m$ since there is no $m$ 
above row $k$ in column 2.
\item 
Only need to check 
there is no violations of condition 4
in column 1. 
Let $i<j$ be a violation.
Then $m$ must be $i$ or $j$.
If $i = m$, 
then $m, j$ would have violated condition 4 before the move.
Now assume $j = m$.
If $i$ were above $m$ in $\Label_\alpha(D)$,
then $i, m$ would have violated condition 4 before the move.
On the other hand, if $i$ were below $m$ in $\Label_\alpha(D)$,
then there is an $i$ above row $k$
in column 2,
so $i, m$ cannot be a violation. 
\end{enumerate}

Thus, after moving one cell 
down to $(1,k)$ in $D$,
the resulting diagram is still a Kohnert diagram,
so $\flat_k(D)$ is well-defined.
\end{proof}

\begin{rem}
In the previous proof of well-definedness, 
we choose the cell containing $m$ and move it down to row $k$.
The resulting diagram is still in $\KD(\alpha)$.
Notice that this might not be the lowest cell
that can do this job. 
See the following example. 
\end{rem}

\begin{exa} Following Example~\ref{eg_sharp}.
We would like to compute $\flat_1(D')$.
In $D'$, there are 4 cells in column 1
above row 1 and there are 2 cells in column 2
above row 1. 
Thus, the condition in Lemma~\ref{L: flat well-defined}.
is satisfied.
We want to
check $\flat_1(D')$ is well-defined.
The proof of well-definedness gives $m = 6$.
After moving the 6 to row 1, 
the resulting filling is in $\KT(\alpha)$,
which implies the underlying diagram is 
in $\KD(\alpha)$.
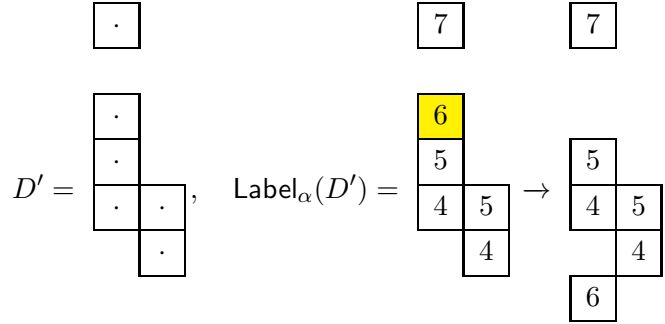
\[
D' = 
\raisebox{2cm}{
\begin{ytableau}
\cdot \\
\none \\
\cdot\\
\cdot \\
\cdot & \cdot \\
\none & \cdot \\
\none 
\end{ytableau}}\,,  \quad
\Label_\alpha(D') = \raisebox{2cm}{
\begin{ytableau}
7 \\
\none \\
*(yellow)6\\
5 \\
4 & 5 \\
\none & 4 \\
\none 
\end{ytableau}}\,\xrightarrow[]{}
\raisebox{2cm}{
\begin{ytableau}
7 \\
\none \\
\none \\
5\\
4 & 5 \\
\none & 4 \\
6
\end{ytableau}}
\]
However, moving the cell $(1,3)$ 
to $(1,1)$ in $D'$ will also make the resulting
diagram in $\KD(\alpha)$.
\end{exa}

\bigskip

Similar to $\sharp_g$, 
the $\flat_k$ operator can commute under certain condition:
\begin{lemma}
\label{L: flat commute}
Take $D \in \KD(\alpha)$.
Take $k_1, k_2 \in [n]$ with $k_1 > k_2$.
Assume $\flat_{k_1}(D) = (D^1, g_1)$ and
$\flat_{k_2}(D^1) = (D^{final}, g_2)$.
If $g_1 < g_2$,
then the two operators ``commute''.
That is:
\begin{enumerate}
\item[$\bullet$] $\flat_{k_2}(D) = (D^2, g_2)$ 
for some $D^2 \in \KD(\alpha)$, and
\item[$\bullet$] $\flat_{k_1}(D^2) = (D^{final}, g_1)$.
\end{enumerate}
\end{lemma}
\begin{proof}
It follows directly 
from Lemma~\ref{L: sharp commute 2} and 
Lemma~\ref{L: sharp and flat are inverses}.
\end{proof}

\subsection{Relations between $\sharp_g$ and $\flat_e$}
In this section, 
we investigate the relationship 
between the two operators introduced above.
We already know the effect of $\sharp_g$ 
can be reversed by the $\flat_e$ operator, and vice versa. 
Next, we show that a sequence of $\sharp_g$
can also be reversed by a sequence of $\flat_e$, and vice versa.

\begin{lemma}
\label{L: many flat after many sharp}
Let $D^0 = (K_1, \emptyset, d) \in \KD(\alpha)$,
and $1 < g_1 < g_2 < \dots < g_m\leqslant n$
with $g_i \notin K_1$.
For $i = 1, 2, \dots, m$,
compute $\sharp_{g_i}(D^{i-1}) = (D^i, k_i)$.
Assume $D^1, \dots, D^m$ are all well-defined.

Find the permutation $\sigma$
such that $k_{\sigma(1)} < \cdots < k_{\sigma(m)}$ and define $T^m = D^m$.
For $i = m, m-1, \dots, 1$,
compute $\flat_{k_{\sigma(i)}}(T^i) 
= (T^{i - 1}, g_i')$.
Then $T^0 = D^0$ and $\{g_1, \dots, g_m\} = 
\{ g_1', \dots, g_m'\}$.
\end{lemma}
\begin{proof}
We may represent $D^0, \dots, D^m$
using the following diagram:
\[
D^0 \xrightarrow[k_1]{\sharp_{g_1}} D^1 \xrightarrow[k_2]{\sharp_{g_2}} D^2 \xrightarrow[k_3]{\sharp_{g_3}} \dots \xrightarrow[k_m]{\sharp_{g_m}} D^m\,.
\]
We put the operator above the arrow
and put the second output under the arrow.

Suppose we find $k_i > k_{i+1}$.
By Lemma~\ref{L: sharp commute 1},
we can swap the order of $\sharp_{g_i}$
and $\sharp_{g_{i+1}}$,
not affecting the last diagram $D^m$.
Thus, after sorting the output numbers
into increasing order, we have

\[
D^0 \xrightarrow[k_{\sigma(1)}]{\sharp_{g_{\sigma(1)}}} \tilde{D}^1 \xrightarrow[k_{\sigma(2)}]{\sharp_{g_{\sigma(2)}}} \tilde{D}^2 \xrightarrow[k_{\sigma(3)}]{\sharp_{g_{\sigma(3)}}} \dots \xrightarrow[k_{\sigma(m)}]{\sharp_{g_{\sigma(m)}}} D^m\,,
\]
where $\tilde{D}^i$ are some diagrams
in $\KD(\alpha)$.
Finally, we have 
\[
D^0 \xleftarrow[g_{\sigma(1)}]{\flat_{k_{\sigma(1)}}} \tilde{D}^1 \xleftarrow[g_{\sigma(2)}]{\flat_{k_{\sigma(2)}}} \tilde{D}^2 \xleftarrow[g_{\sigma(3)}]{\flat_{k_{\sigma(3)}}} \dots \xleftarrow[g_{\sigma(m)}]{\flat_{k_{\sigma(m)}}} D^m\,.
\]

By $T^m = D^m$, we have $D^0 = T^0$.
\end{proof}

\begin{exa} Consider $\alpha = (0,0,2,0,3,1,2)$, $g_1=3, g_2 = 5, g_3 = 6$ and $D^0 \in \KD(\alpha)$. Starting with $D^0$, we compute $\sharp_{g_1}, \sharp_{g_2}$ and then $\sharp_{g_3}$ to obtain $D^3$.
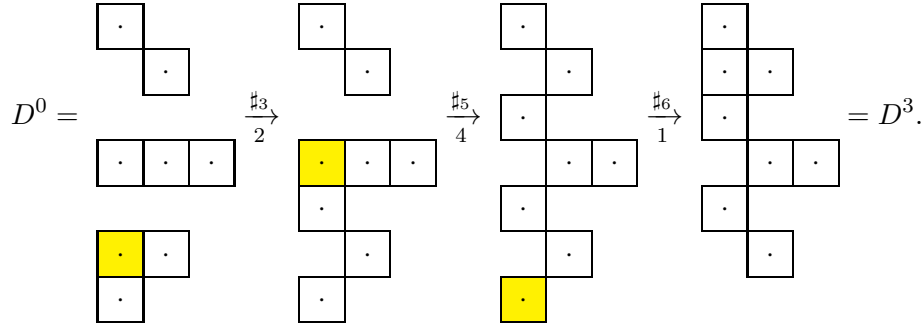
\[
D^0 = 
\raisebox{1cm}{
\begin{ytableau}
\cdot \\
\none & \cdot \\
\none \\
\cdot & \cdot & \cdot\\
\none \\
*(yellow)\cdot & \cdot \\
\cdot
\end{ytableau}}\xrightarrow[2]{\sharp_{3}}
\raisebox{1cm}{
\begin{ytableau}
\cdot \\
\none & \cdot \\
\none \\
*(yellow)\cdot & \cdot & \cdot\\
\cdot \\
\none & \cdot \\
\cdot
\end{ytableau}}\xrightarrow[4]{\sharp_{5}}
\raisebox{1cm}{
\begin{ytableau}
\cdot \\
\none & \cdot \\
\cdot \\
\none & \cdot & \cdot\\
\cdot \\
\none & \cdot \\
*(yellow)\cdot
\end{ytableau}} \xrightarrow[1]{\sharp_{6}}
\raisebox{1cm}{
\begin{ytableau}
\cdot \\
\cdot & \cdot \\
\cdot \\
\none & \cdot & \cdot\\
\cdot \\
\none & \cdot \\
\none
\end{ytableau}}
= D^3.
\]
We obtain $k_1 = 2, k_2 = 4$ and $k_3 = 1$, which are highlighted in the above figure. We can pick the permutation $\sigma$ with one-line notation $312$ and obtain $k_{\sigma(1)} = 1, k_{\sigma(2)} = 2$ and $k_{\sigma(3)} = 4$. Then we have $g_{\sigma(1)} = 6, g_{\sigma(2)} = 3$ and $g_{\sigma(3)} = 5$, which yields the following sequence of operation with the same $D^3$ as the final output.
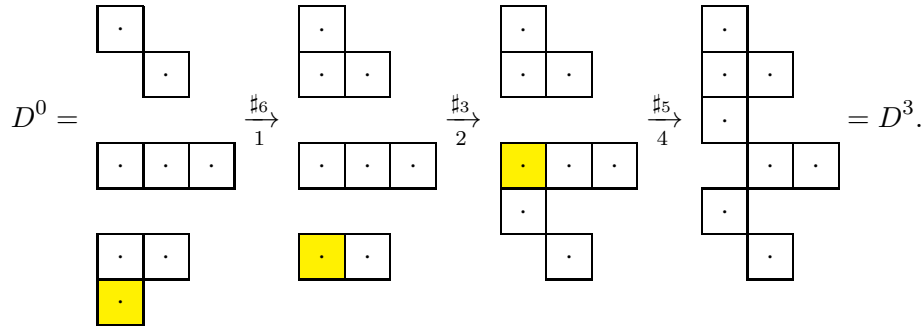
\[
D^0 = 
\raisebox{1cm}{
\begin{ytableau}
\cdot \\
\none & \cdot \\
\none \\
\cdot & \cdot & \cdot\\
\none \\
\cdot & \cdot \\
*(yellow)\cdot
\end{ytableau}}\xrightarrow[1]{\sharp_{6}}
\raisebox{1cm}{
\begin{ytableau}
\cdot \\
\cdot & \cdot \\
\none \\
\cdot & \cdot & \cdot\\
\none \\
*(yellow)\cdot & \cdot \\
\none
\end{ytableau}}\xrightarrow[2]{\sharp_{3}}
\raisebox{1cm}{
\begin{ytableau}
\cdot \\
\cdot & \cdot \\
\none \\
*(yellow)\cdot & \cdot & \cdot\\
\cdot \\
\none & \cdot \\
\none
\end{ytableau}}\xrightarrow[4]{\sharp_{5}}
\raisebox{1cm}{
\begin{ytableau}
\cdot \\
\cdot & \cdot \\
\cdot \\
\none & \cdot & \cdot\\
\cdot \\
\none & \cdot \\
\none
\end{ytableau}} = D^3.
\]
Applying $\flat_4, \flat_2$ then $\flat_1$ on $D^3$, we will recover $D^0$.
\end{exa}
\begin{lemma}
\label{L: many sharp after many flat}
Let $D^0 = (K_1, \emptyset, d) \in \KD(\alpha)$, 
and $n >  e_1 > e_2 > \dots > e_m\geqslant 1$,
with $e_i \notin K_1$.
For $i = 1, 2, \dots, m$,
compute $\flat_{e_i}(D^{i-1}) = (D^i, k_i)$.
Assume $D^1, \dots, D^m$ are all well-defined.

Find the permutation $\sigma$
such that $k_{\sigma(1)} > \cdots > k_{\sigma(m)}$.
Now define $T^m = D^m$.
For $i = m, m-1, \dots, 1$,
compute $\sharp_{k_{\sigma(i)}}(T^i) 
= (T^{i - 1}, e_i')$.
Then $T^0 = D^0$ and $\{e_1, \dots, e_m\} =
\{ e_1', \dots, e_m'\}$.
\end{lemma}
\begin{proof}
The proof is the same as the previous proof, using Lemma~\ref{L: flat commute}
instead of Lemma~\ref{L: sharp commute 1}.
\end{proof}

\section{Recursive descriptions of the maps}\label{section.recursive}
We have described our maps $\Psi_\alpha$ and $\Phi_\alpha$ 
via $\sharp_G(K)$ and $\flat_E(L)$ 
in section~\ref{section.maps}. These descriptions are  simple to state but 
hard to work with.
Now, we will describe the $\Psi_\alpha$ 
and $\Phi_\alpha$ recursively,
involving definitions from section~\ref{section.flat.sharp}
and section~\ref{section.recursive}.
Using the new alternative descriptions,
we can establish Lemma~\ref{L: Psi well-defined}, 
Lemma~\ref{L: Phi well-defined} and Theorem~\ref{T: Main}.

\subsection{Recursive description
of $\Psi_\alpha$}

Let $G$ be an arbitrary diagram.
First, we can recursively describe the operator $\sharp_G(\cdot)$.
If $D$ is empty, then $\sharp_G(D)$ is also 
empty if $G = \emptyset$,
or undefined otherwise.
If $D$ is not empty, 
write $D$ as $(K_1, \emptyset, d)$.
Let $G_{\geqslant 2}$ be the diagram $\{ (c-1, r): (c,r) \in G, c \geqslant 2\}$.
View $d$ as an element of $\KD(\overline{M(\alpha, K_1)})$
and find $d' = \sharp_{G_{\geqslant 2}}(d)$ recursively.
Let $G_1$ be the set $\{r: (1, r) \in G\}$
and assume $G_1 = \{ g_1 < \cdots < g_{|G_1|}\}$.
Let $D^0$ be the Kohnert diagram $(K_1, \emptyset, d')$.
Then compute $\sharp_{g_i}(D^{i-1}) = (D^i, k_i)$ for $1\leqslant i \leqslant |G_1|$.
The final output is $D^{|G_1|}$.

\begin{lemma}
The description is equivalent to the description of $\sharp_G(K)$ 
in subsection~\ref{prelim.kkd.to.rsvt}.
\end{lemma}

\begin{proof}
Recall that $\sharp_G(K)$ iterates over cells of $G$
from right to left. 
Within each column, it goes from bottom to top.
For a cell $(c,r) \in G$,
it picks the highest cell weakly below $(c,r)$ such that once 
this cell is raised to $(c,r)$, 
the diagram is still in $\KD(\alpha)$.
Then it moves the chosen cell to $(c,r)$.

Let $D = (K_1, \emptyset, d)$ be the Kohnert diagram
at the beginning of the iteration of $(c,r) \in G$.
Assume $c \geqslant 2$.
By the recursive description of $\KD(\alpha)$,
the following two statements are equivalent:
\begin{enumerate}
\item[$\bullet$] $(c, r')$ is a cell in $D$ such that if we move 
it to $(c,r)$,
the diagram is still in $\KD(\alpha)$.
\item[$\bullet$] $(c-1, r')$ is a cell in $d$ such that if we move 
it to $(c-1,r)$,
the diagram is still in $\KD(\overline{M(\alpha, K_1)})$.
\end{enumerate}
Thus, iterations of $(c,r) \in G$ with $c \geqslant 2$
will behave the same as if $\sharp_{G_\geqslant 2}$ acts 
on $d \in \KD(\overline{M(\alpha, K_1)})$.
Then iterations of $(1,g) \in G$
can be characterized by the $\sharp_g$ operator. 
\end{proof}

Now we can recursively describe the map $\Psi_\alpha$.
To make our description concise, 
we extend $\sharp_g(\cdot)$ to diagram pairs
$(K, G)$ such that $(K, \emptyset) \in \KD(\alpha)$ 
and $G$ has no cells in column 1.
The operator $\sharp_g(\cdot)$ acts
as if acting on $(K, \emptyset)$.

Now take $D = (K_1, G_1, d) \in \KKD(\alpha)$.
If $D$ is the empty pair, 
we have $\Psi_\alpha(D) = D$.  
Otherwise, 
let $t = \Psi_{\gamma}(d)$, 
where $\gamma = \overline{M(\alpha, K_1)}$.
Assume $G_1 = \{g_1 <  \dots < g_{|G_1|}\}$.
Let $D^0$ be the diagram pair $(K_1, \emptyset, t)$.
Then compute $\sharp_{g_i}(D^{i-1}) = (D^i, k_i)$
and write $D^i$ as $(K^i_1, \emptyset, t)$.
Finally, $\Psi_\alpha(D)$ is 
$(K^{|G_1|}_1, \{k_1, \dots, k_{|G_1|}\}, t)$.

\begin{exa} Consider $\alpha = (0,0,2,0,3,1,2)$.
Let $D = (K, G)$ be the following element in $\KKD(\alpha)$.

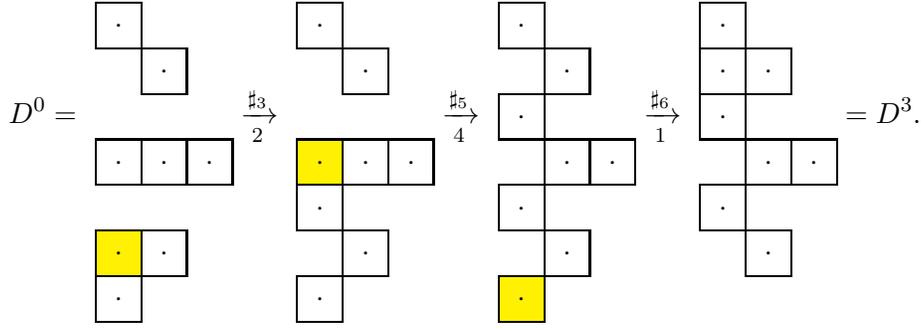
\[
D = 
\raisebox{1cm}{
\begin{ytableau}
\cdot \\
\X & \X \\
\X & \cdot\\
\cdot & \cdot & \cdot\\
\X \\
\cdot & \cdot \\
\cdot
\end{ytableau}}
\]
If we compute $\Psi_\alpha(D)$
using the description in subsection~\ref{prelim.kkd.to.rsvt},
we would go through the following iterations. 
\[
K = 
\raisebox{1cm}{
\begin{ytableau}
\cdot \\
\none \\
\none & *(yellow)\cdot\\
\cdot & \cdot & \cdot\\
\none \\
\cdot & \cdot \\
\cdot
\end{ytableau}}\xrightarrow[]{(2,6)}
\raisebox{1cm}{
\begin{ytableau}
\cdot \\
\none & \cdot \\
\none \\
\cdot & \cdot & \cdot\\
\none \\
*(yellow)\cdot & \cdot \\
\cdot
\end{ytableau}}\xrightarrow[]{(1,3)}
\raisebox{1cm}{
\begin{ytableau}
\cdot \\
\none & \cdot \\
\none \\
*(yellow)\cdot & \cdot & \cdot\\
\cdot \\
\none & \cdot \\
\cdot
\end{ytableau}}\xrightarrow[]{(1,5)}
\raisebox{1cm}{
\begin{ytableau}
\cdot \\
\none & \cdot \\
\cdot \\
\none & \cdot & \cdot\\
\cdot \\
\none & \cdot \\
*(yellow)\cdot
\end{ytableau}} \xrightarrow[]{(1,6)}
\raisebox{1cm}{
\begin{ytableau}
\cdot \\
\cdot & \cdot \\
\cdot \\
\none & \cdot & \cdot\\
\cdot \\
\none & \cdot \\
\none
\end{ytableau}}
= L.
\]
Thus, we have
\[
\Psi_\alpha(D) =
(L, (K \sqcup G) - L) = 
\raisebox{1cm}{
\begin{ytableau}
\cdot \\
\cdot & \cdot \\
\cdot & \X \\
\X & \cdot & \cdot\\
\cdot \\
\X & \cdot \\
\X
\end{ytableau}}
\]
Now we try our new recursive description. 
We may write $D$ as $(K_1, G_1, d)$,
where $K_1 = \{1,2,4,7\}, G_1 = \{3,5,6\}$
and $d$ is illustrated below.
Our new description would first
view $d$ as an element of $\KKD(\overline{M(\alpha, K_1)}) = 
\KKD((0,1,0,2,0,0,1))$
and send it to $t$:
\[
d = 
\raisebox{1cm}{
\begin{ytableau}
\none \\
\X \\
\cdot\\
\cdot & \cdot\\
\none \\
\cdot \\
\none 
\end{ytableau}}\xrightarrow[]{}
\raisebox{1cm}{
\begin{ytableau}
\none \\
\cdot \\
\X \\
\cdot & \cdot\\
\none \\
\cdot \\
\none 
\end{ytableau}} = t.
\]

It remains to perform $\sharp_3$, $\sharp_5$
and $\sharp_6$.

\[
D^0 = 
\raisebox{1cm}{
\begin{ytableau}
\cdot \\
\none & \cdot \\
\none \\
\cdot & \cdot & \cdot\\
\none \\
*(yellow)\cdot & \cdot \\
\cdot
\end{ytableau}}\xrightarrow[2]{\sharp_{3}}
\raisebox{1cm}{
\begin{ytableau}
\cdot \\
\none & \cdot \\
\none \\
*(yellow)\cdot & \cdot & \cdot\\
\cdot \\
\none & \cdot \\
\cdot
\end{ytableau}}\xrightarrow[4]{\sharp_{5}}
\raisebox{1cm}{
\begin{ytableau}
\cdot \\
\none & \cdot \\
\cdot \\
\none & \cdot & \cdot\\
\cdot \\
\none & \cdot \\
*(yellow)\cdot
\end{ytableau}} \xrightarrow[1]{\sharp_{6}}
\raisebox{1cm}{
\begin{ytableau}
\cdot \\
\cdot & \cdot \\
\cdot \\
\none & \cdot & \cdot\\
\cdot \\
\none & \cdot \\
\none
\end{ytableau}}
= D^3.
\]
Finally, the image is just 
$(\{3,5,6,7\}, \{1,2,4\}, t)$,
which agrees with the computation above.
\end{exa}

It is clear that this recursive description
agrees with the original description of $\Psi_\alpha$.
To prove Lemma~\ref{L: Psi well-defined},
we need to show $t, D^0, \dots, D^{|G_1|}$
exist and satisfy our assumptions. 
Besides, we need to check the final output
is a diagram pair in $\RSVT(\alpha)$.

\begin{proof}[Proof of Lemma~\ref{L: Psi well-defined}]
Prove by induction on $\max(\alpha)$.
We may assume $\Psi_\gamma$ is a well-defined map
from $\KKD(\gamma)$ to $\RSVT(\gamma)$,
where $\gamma = \overline{M(\alpha, K_1)}$
Thus, we know $t \in \RSVT(\gamma)$.

Then clearly if we ignore ghost cells in $D^0$,
it is in $\KD(\alpha)$. 
Moreover, $D^0$ has no ghost cells in column 1. 
Next, we need to show the diagram pairs $D^i$ are well-defined.
By Theorem~\ref{T: KKD image},
we know for each $g_i$, 
$|[g_i, n] \cap \supp(\alpha)| > |[g_i, n] \cap K_1|$.
Notice that the first $i-1$ iterations will not
move any cells above row $g_{i-1}$.
Thus, $[g_i, n] \cap K_1 = [g_i, n] \cap K^{i-1}_1$.
By Lemma~\ref{L: sharp well-defined},
$D^i$ exists. 

Next, we need to check 
the image is in $\RSVT(\alpha)$.
In other words, 
$T = (K^{|G_1|}, \{k_1, \dots, k_{|G_1|}\}, t)$
should satisfy all four conditions in Theorem~\ref{T: RSVT image}. 
Let $L_1'$ be the set of 
row indices of Kohnert cells in column 1 of $t$.
\begin{enumerate}
    \item The first condition is immediate.
    \item Since Kohnert cells of 
$D^{|G_1|} = (K_1^{|G_1|}, \emptyset, t)$ 
is in $\KD(\alpha)$,
we have $K_1^{|G_1|} \leqslant \supp(\alpha)$.
    \item For each $k_i$, we show
$|(k_i, n] \cap K_1^{|G_1|}| 
> |(k_i, n] \cap L_1'|$.
Since Kohnert cells of $D^{i-1}
= (K_1^{i-1}, \emptyset, t)$ 
is in $\KD(\alpha)$,
we know the Kohnert cells of $t$
is in $\KD(\overline{M(\alpha, K_1^{i-1}})$.
Thus, $L_1' \leqslant \supp(\overline{M(\alpha, K_1^{i-1}}) 
\subseteq K_1^{i-1}$.
We have
$$
|(k_i, n] \cap K_1^{i-1}| 
\geqslant |(k_i, n] \cap L_1'|.
$$
Since $K_1^{i}$ is obtained from $K_1^{i-1}$
by replacing $k_i$ with a larger number, 
we have $|(k_i, n] \cap K_1^{i}| =|(k_i, n] \cap K_1^{i-1}|+1
> |(k_i, n] \cap L_1'|$.
To obtain $K_1^{|G_1|}$,
we replace each of $k_{i+1}, \dots, k_{|G_1|}$
in $K_1^{i}$ with a larger number.
Therefore, $|(k_i, n] \cap K_1^{|G_1|}| 
> |(k_i, n] \cap L_1'|$.
\item Kohnert cells of $t$
is in $\KD(\overline{M(\alpha, K_1^{|G_1|})})$ and 
$t \in \RSVT(\overline{M(\alpha, K_1)})$.
By Lemma~\ref{L: RSVT in another set},
$t \in \RSVT(\overline{M(\alpha, K_1^{|G_1|})})$.
\end{enumerate}
\end{proof}

\subsection{Recursive description of $\Phi_\alpha$}

Let $E$ be an arbitrary diagram.
We can recursively describe the operator $\flat_E(\cdot)$.
If $D$ is empty, then $\flat_E(D)$ is also 
empty if $E = \emptyset$, or undefined otherwise.
If $D$ is not empty, 
write $D$ as $(L_1, \emptyset, t)$. Let $E_{\geqslant 2}$ be the diagram $\{ (c-1, r): (c,r) \in E, c \geqslant 2\}$.
Let $E_1$ be the set $\{r: (1, r) \in E\}$. Assume $E_1 = \{ e_1 > \cdots > e_{|E_1|}\}$.
Let $D^0$ be the Kohnert diagram $(L_1, \emptyset, t)$.
Then compute $\flat_{e_i}(D^{i-1}) = (D^i, g_i)$ for $1\leqslant i \leqslant |E_1|$ and write $D^i$ as $(L^i_1, \emptyset, t)$.
View $t$ as an element of 
$\KD(\overline{M(\alpha, L_1^{|E_1|})})$
and find $t' = \flat_{E_{\geqslant 2}}(t)$ recursively.
Finally, $\flat_E(D) = (L_1^{|E_1|}, \emptyset, t')$.

\begin{lemma}
The description is equivalent to the description of $\flat_E(L)$ 
in subsection~\ref{prelim.rsvt.to.kkd}.
\end{lemma}

\begin{proof}
Recall that $\flat_E(L)$ iterates over cells of $E$
from left to right.  Within each column, it goes from top to bottom.
For a cell $(c,r) \in E$,
it picks the lowest cell weakly above $(c,r)$ such that once 
this cell is lowered to $(c,r)$, 
the diagram is still in $\KD(\alpha)$.
Then it moves the chosen cell to $(c,r)$.

Let $D = (L_1, \emptyset, t)$ be the Kohnert diagram
at the beginning of the iteration of $(c,r) \in E$.
The iterations of $(1,e) \in E$
can be characterized by the $\flat_e$ operator.
Assume $c \geqslant 2$.
By the recursive description of $\KD(\alpha)$,
the following two statements are equivalent for any $r'<r$:
\begin{enumerate}
\item[$\bullet$] $(c, r')$ is a cell in $D$ such that if we move 
it to $(c,r)$,
the diagram is still in $\KD(\alpha)$.
\item[$\bullet$] $(c-1, r')$ is a cell in $t$ such that if we move 
it to $(c-1,r)$,
the diagram is still in $\KD(\overline{M(\alpha, L_1^{|E_1|})})$.
\end{enumerate}
Thus, iterations of $(c,r) \in E$ with $c \geqslant 2$
will behave the same as if $\flat_{E_\geqslant 2}$ acts 
on $t \in \KD(\overline{M(\alpha, L_1^{|E_1|})})$.
\end{proof}

Now we can recursively describe the map $\Phi_\alpha$.
To make our description concise, 
we extend $\flat_e(\cdot)$ to diagram pairs
$(K, G)$ such that $(K, \emptyset) \in \KD(\alpha)$ 
and $G$ has no cells in column 1.
The operator $\flat_e(\cdot)$ acts
as if acting on $(K, \emptyset)$.

Now take $D = (L_1, E_1, t) \in \RSVT(\alpha)$.
If $D$ is the empty diagram pair, 
we have $\Phi_\alpha(D) = D$.  
Otherwise, assume $E_1 = \{e_1 >  \dots > e_{|E_1|}\}$.
Let $D^0$ be the diagram pair $(L_1, \emptyset, t)$.
Then compute $\flat_{e_i}(D^{i-1}) = (D^i, g_i)$
and write $D^i$ as $(L^i_1, \emptyset, t)$. 
Notice that $t \in \RSVT(\overline{M(\alpha, L_1)})$
and its Kohnert cells is in $\KD(\overline{M(\alpha, L^{|E_1|}_1)})$.
Thus, by Lemma~\ref{L: RSVT in another set}, 
we may view $t$ as an element of 
$\RSVT(\gamma)$, 
where $\gamma = \overline{M(\alpha, L^{|E_1|}_1)}$.
Let $d = \Phi_\gamma(t)$. 
Finally, $\Phi_\gamma(T) = (L^{|E_1|}_1, \{g_1, \dots, g_{|E_1|}\}, d)$.

\begin{exa} Consider $\alpha = (0,0,2,0,3,1,2)$.
Let $D = (L, E)$ be the following element in $\RSVT(\alpha)$.

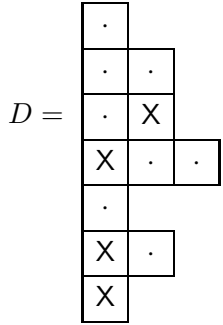
\[
D = 
\raisebox{1cm}{
\begin{ytableau}
\cdot \\
\cdot & \cdot \\
\cdot & \X \\
\X & \cdot & \cdot\\
\cdot \\
\X & \cdot \\
\X
\end{ytableau}}
\]
If we want to compute $\Phi_\alpha(D)$
using the description in subsection~\ref{prelim.rsvt.to.kkd},
we would go through the following iterations. 
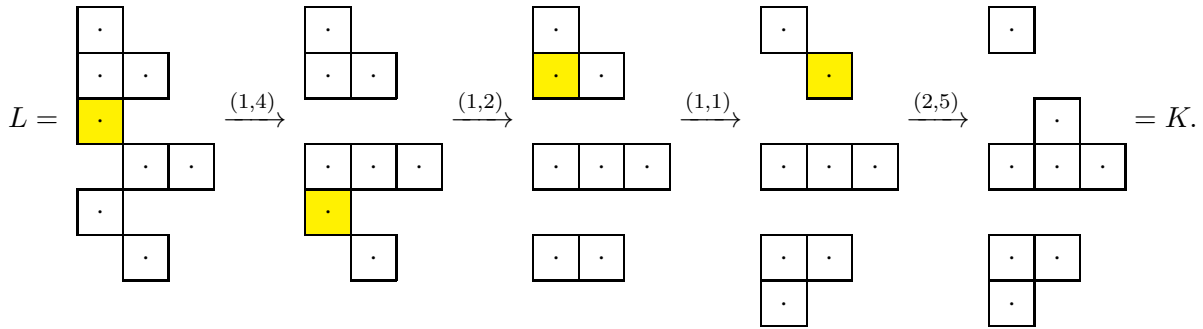
\[
L = 
\raisebox{1cm}{
\begin{ytableau}
\cdot \\
\cdot & \cdot \\
*(yellow)\cdot \\
\none & \cdot & \cdot\\
\cdot \\
\none & \cdot \\
\none
\end{ytableau}}
\xrightarrow[]{(1,4)}
\raisebox{1cm}{
\begin{ytableau}
\cdot \\
\cdot & \cdot \\
\none \\
\cdot & \cdot & \cdot\\
*(yellow)\cdot \\
\none & \cdot \\
\none
\end{ytableau}}
\xrightarrow[]{(1,2)}
\raisebox{1cm}{
\begin{ytableau}
\cdot \\
*(yellow) \cdot & \cdot \\
\none \\
\cdot & \cdot & \cdot\\
\none \\
\cdot & \cdot \\
\none
\end{ytableau}}
\xrightarrow[]{(1,1)}
\raisebox{1cm}{
\begin{ytableau}
\cdot \\
\none & *(yellow) \cdot \\
\none \\
\cdot & \cdot & \cdot\\
\none \\
\cdot & \cdot \\
\cdot
\end{ytableau}}
\xrightarrow[]{(2,5)}
\raisebox{1cm}{
\begin{ytableau}
\cdot \\
\none \\
\none & \cdot\\
\cdot & \cdot & \cdot\\
\none \\
\cdot & \cdot \\
\cdot
\end{ytableau}} = K.
\]
Thus, $\Phi_\alpha(D)$ is:
\[
(K, (L \sqcup E) - K) = 
\raisebox{1cm}{
\begin{ytableau}
\cdot \\
\X & \X \\
\X & \cdot\\
\cdot & \cdot & \cdot\\
\X \\
\cdot & \cdot \\
\cdot
\end{ytableau}}
\]
Now we try our new recursive description. 
We may write $D$ as $(L_1, E_1, t)$,
where $L_1 = \{3,5,6,7\}, E_1 = \{1,2,4\}$
and $t$ is illustrated below.
Our new description would first
perform $\flat_4, \flat_2$ 
and $\flat_1$ on $D^0$:

\[
D^0 = 
\raisebox{1cm}{
\begin{ytableau}
\cdot \\
\cdot & \cdot \\
*(yellow)\cdot & \X \\
\none & \cdot & \cdot\\
\cdot \\
\none & \cdot \\
\none
\end{ytableau}}
\xrightarrow[5]{\flat_4}
\raisebox{1cm}{
\begin{ytableau}
\cdot \\
\cdot & \cdot \\
\none & \X \\
\cdot & \cdot & \cdot\\
*(yellow)\cdot \\
\none & \cdot \\
\none
\end{ytableau}}
\xrightarrow[3]{\flat_2}
\raisebox{1cm}{
\begin{ytableau}
\cdot \\
*(yellow) \cdot & \cdot \\
\none & \X \\
\cdot & \cdot & \cdot\\
\none \\
\cdot & \cdot \\
\none
\end{ytableau}}
\xrightarrow[6]{\flat_1}
\raisebox{1cm}{
\begin{ytableau}
\cdot \\
\none & \cdot \\
\none & \X \\
\cdot & \cdot & \cdot\\
\none \\
\cdot & \cdot \\
\cdot
\end{ytableau}}
= D^3.
\]

Now, view $t$ as an element of $\KKD(\overline{M(\alpha, L_1^3)}) = 
\KKD((0,1,0,2,0,0,1))$
and send it to $d$:
\[
t = 
\raisebox{1cm}{
\begin{ytableau}
\none \\
\cdot \\
\X \\
\cdot & \cdot\\
\none \\
\cdot \\
\none 
\end{ytableau}}
\xrightarrow[]{}
\raisebox{1cm}{
\begin{ytableau}
\none \\
\X \\
\cdot\\
\cdot & \cdot\\
\none \\
\cdot \\
\none 
\end{ytableau}}= d.
\]
Finally, $\Phi_\alpha(D) =  
(\{1,2,4,7\}, \{3,5,6\}, d)$,
which agrees with the computation above.
\end{exa}

It is clear that this recursive description
agrees with the original description of $\Phi_\alpha$.
To prove Lemma~\ref{L: Phi well-defined},
we need to show $d, D^0, \dots, D^{|E_1|}$
exist and satisfy our assumptions. 
Moreover, we need to check the final output
is a diagram pair in $\KKD(\alpha)$.

\begin{proof}[Proof of Lemma~\ref{L: Phi well-defined}]
Prove by induction on $\max(\alpha)$.
We may assume $\Phi_\gamma$ a well-defined
map from $\RSVT(\gamma)$ to $\KKD(\gamma)$
for any $\gamma$ with
$\max(\gamma) < \max(\alpha)$.

Clearly $D^0$ is a diagram pair
whose Kohnert cells form a diagram in 
$\KD(\alpha)$
and has no ghost cells in column 1. 
Next, we need to show $D^i$ is well-defined for an arbitrary 
$1 \leqslant i \leqslant |E_1|$.
Notice that the first $i-1$ iterations will not
move any cells weakly below row $e_{i-1}$. 
Let $L_1'$ consists of row indices of Kohnert cells
in column 1 of $t$.
Thus,
\begin{align*}
|(e_{i-1},n] \cap L_1^{i-1}| = |(e_{i-1},n] \cap  L_1| 
 > |(e_{i-1},n] \cap  L_1'|,
\end{align*}
where the inequality
follows from Theorem~\ref{T: RSVT image}.
By Lemma~\ref{L: flat well-defined},
$D^i$ exists. 
Finally, by the inductive hypothesis, 
$d \in \KKD(\overline{M(\alpha, L^{|E_1|}_1)})$.

Next, we need to check 
the final image is in $\KKD(\alpha)$.
In other words, 
$(L^{|E_1|}, \{g_1, \dots, g_{|E_1|}\}, d)$
should satisfy all four conditions in Theorem~\ref{T: KKD image}.
\begin{enumerate}
    \item The first condition is immediate.
    \item Since Kohnert cells of 
$D^{|E_1|} = (L_1^{|E_1|}, \emptyset, t)$ 
is in $\KD(\alpha)$,
we have $L_1^{|E_1|} \leqslant \supp(\alpha)$.
    \item  For each $g_i$, we show
$|[g_i, n] \cap \,\supp(\alpha)| 
> |[g_i, n] \cap\, L^{|E_1|}_1|$. Since $(L_1,E_1,t)\in \RSVT(\alpha)$, we have $L_1 \leqslant \supp(\alpha)$. 
Since $L_1^{j}$ is obtained from $L^{j-1}_1$ by replacing $g_j$ with a smaller number,
$L_1^{|E_1|} \leqslant \cdots \leqslant L_1^0 \leqslant \supp(\alpha)$.
By Lemma~\ref{lem:subset_order}, 
$$|[g_i,n]\cap L_1^{|E_1|}| \leqslant \cdots \leqslant|[g_i,n]\cap L_1^0| \leqslant |[g_i,n]\cap \,\supp(\alpha)|.$$
Notice that
$|[g_i,n]\cap L^{i}_1| = |[g_i,n]\cap L^{i-1}_1|-1 $. 
Thus, $|[g_i, n] \cap\, L^{|E_1|}_1| 
< |[g_i, n] \cap \,\supp(\alpha)|$.
    \item This is checked above.
\end{enumerate}
\end{proof}

\subsection{Proof of Theorem~\ref{T: Main}}

In this subsection, we prove Theorem~\ref{T: Main}.

\begin{proof}[Proof of Theorem~\ref{T: Main}]
The maps clearly preserve $\wt(\cdot)$ and $\ex(\cdot)$.
To show they are mutually inverses, 
we only need to check the following two statements. 
\begin{enumerate}
    \item Take $D \in \KKD(\alpha)$.
    Let $T = \Psi_\alpha(D) \in \RSVT(\alpha)$.
    Then $\Phi_\alpha(T) = D$.
    \item Take $T \in \RSVT(\alpha)$.
    Let $D = \Phi_\alpha(D) \in \KKD(\alpha)$.
    Then $\Psi_\alpha(D) = T$.
\end{enumerate}

We only establish the first statement 
using Lemma~\ref{L: many flat after many sharp}. 
The second statement can be proved similarly
using Lemma~\ref{L: many sharp after many flat} instead.

We prove by induction on $\max(\alpha)$.
When $\max(\alpha) = 0$,
$D = (\emptyset, \emptyset)$ and our claim is immediate. 

Now assume $\max(\alpha) > 0$.
Let $D = (K_1, G_1, d)$. 
First, we compute $\Psi_\alpha(D)$
using our recursive description.
Let $t = \Psi_\gamma(d)$ where $\gamma = \overline{M(\alpha,K_1)}$.
Assume $G_1 = \{g_1< g_2< \dots <g_{|G_1|}\}$.
Let $D^0 = (K_1,\emptyset,t)$ and $\sharp_{g_i}(D^{i-1}) = (D^i,k_i)=((K_1^{i},\emptyset,t),k_i)$ 
for $i = 1, \dots, |G_1|$.
Then we know $D$ is sent to 
$T = (K_1^{|G_1|},E_1, t\})\in \RSVT(\alpha)$,
where $E_1 = \{k_1,\dots, k_{|G_1|}\}$.

Now we compute $\Phi_\alpha(T)$ 
using our recursive description. 
Now write $E_1$ as $\{e_1 > \dots > e_{|G_1|}\}$.
After applying $\flat_{e_1}, \dots, \flat_{e_{|G_1|}}$
on $(K_1^{|G_1|}, \emptyset, t)$,
by Lemma~\ref{L: many flat after many sharp}, 
the resulting diagram pair is $(K_1, \emptyset, t)$
and $G_1$ consists of the output numbers.
Finally, by the inductive hypothesis, 
$\Phi_\gamma(t) = d$.
Thus, $\Phi_\alpha(T) = (K_1, G_1, d) = D$.
\end{proof}

Now we have the desired weight-preserving bijection between $\KKD(\alpha)$ and $\RSVT(\alpha)$.
We can claim the Ross-Yong conjecture is correct.

\begin{cor}
The Lascoux polynomials indexed by $\alpha$, has a combinatorial formula with $\KKD(\alpha)$, i.e.,
\[
\fL^{(\beta)}_\alpha = \sum_{D \in \KKD(\alpha)}\beta^{\ex(D)}\bm{x}^{\wt(D)}\,.
\]
\end{cor}

\section{Acknowledgments}
The authors thank Sami Assaf, Brendon Rhoades 
and Mark Shimozono for helpful
conversations. 
%%%%%%%%%%%%%%%%%%%%%%%%%%%%%%%%%%%%%%%%%%%%%%%%%%%
\bibliographystyle{alpha}
\bibliography{citation}{}
\end{document}